\documentclass[12pt]{amsart}

\usepackage{amssymb,mathrsfs,bm,amsmath,amsthm,color,mathtools}
\usepackage{graphicx,caption} 
\usepackage{float}

\usepackage{enumerate}
\usepackage[centering]{geometry}
\theoremstyle{plain}
\newtheorem{thm}{Theorem}[section]
\newtheorem{defn}{Definition}[section]

\newtheorem{prop}[thm]{Proposition}
\newtheorem{cor}[thm]{Corollary}
\newtheorem{lem}[thm]{Lemma}
\newtheorem{exmp}{Example}
\theoremstyle{remark}
\newtheorem{rem}{Remark}
\numberwithin{equation}{section}

\DeclareMathOperator{\hdim}{\dim_H}

\newcommand{\N}{\mathbb N}
\newcommand{\uk}{\mathcal U^\kappa(x)}
\newcommand{\bk}{\mathcal B^\kappa(x)}
\newcommand{\ud}{\underline d}
\newcommand{\od}{\overline d}
\newcommand{\ur}{\underline{R}}
\newcommand{\ovr}{\overline{R}}
\newcommand{\mi}{\mathcal M_{\mathrm{inv}}}
\newcommand{\R}{\mathbb R}

\newcommand{\uph}{\mu_\phi}

\begin{document}
	\title{Uniform approximation problems of expanding Markov maps}
\author{Yubin He}

\address{Department of Mathematics, South China University of Technology,	Guangzhou, Guangdong 510641, P.~R.\ China}

\email{yubinhe007@gmail.com}

\author{Lingmin Liao}
\address{School of Mathematics and Statistics, Wuhan University, Wuhan, Hubei 430072, P.~R.\ China}
\email{lmliao@whu.edu.cn}


\subjclass[2010]{28A80}

\keywords{Uniform Diophantine approximation, Markov map, Hausdorff dimension, Gibbs measure, multifractal spectrum}
	\begin{abstract}
		Let $ T:[0,1]\to[0,1] $ be an expanding Markov map with a finite partition. Let $ \uph $ be the invariant Gibbs measure associated with a H\"older continuous potential $ \phi $.  In this paper, we investigate the size of the uniform approximation set
		\[\uk:=\{y\in[0,1]:\forall N\gg1,~\exists n\le N, \text{ such that }|T^nx-y|<N^{-\kappa}\},\]
		where $ \kappa>0 $ and $ x\in[0,1] $. The critical value of $ \kappa $ such that $ \hdim\uk=1 $ for $ \uph $-a.e.\,$ x $ is proven to be $ 1/\alpha_{\max} $, where $ \alpha_{\max}=-\int \phi\,d\mu_{\max}/\int\log|T'|\,d\mu_{\max} $ and $ \mu_{\max} $ is the Gibbs measure associated with the potential $ -\log|T'| $. Moreover, when $ \kappa>1/\alpha_{\max} $, we show that for $ \uph $-a.e.\,$ x $, the Hausdorff dimension of $ \uk $ agrees with the multifractal spectrum of $ \uph $. 
	\end{abstract}
	\maketitle

	\section{Introduction and motivation}
	Denote by $ \|\cdot\| $ the distance to the nearest integer. The famous Dirichlet Theorem asserts that for any real numbers $ x\in[0,1] $ and $ N\ge 1 $, there exists a positive integer $ n $ such that
	\begin{equation}\label{eq:Dirichlet's theorem}
		\|nx\|<\frac{1}{N}\quad\text{and}\quad n\le N.
	\end{equation}
	As a corollary,
	for any $ x\in[0,1] $, there exist infinitely many positive integers $ n $ such that
	\begin{equation}\label{eq:Dirichlet's corollary}
		\|nx\|\le\frac{1}{n}.
	\end{equation}
	Let $ (\{nx\})_{n\ge 0} $ be the orbit of $0$ of the irrational rotation by an irrational number $ x $, where $ \{nx\} $ is the fractional part of $ nx $. From the dynamical perspective, Dirichlet Theorem and its corollary describe the rate at which $ 0 $ is approximated by the orbit $ (\{nx\})_{n\ge 0} $ in a uniform and asymptotic way, respectively.
	

	 In general, one can study the Hausdorff dimension of the set of points which are approximated by the orbit $ (\{nx\})_{n\ge 0} $ with a faster speed. 
	 For the asymptotic approximation, Bugeaud~\cite{Bug03} and, independently, Schmeling and Troubetzkoy~\cite{ScTr03} proved that
\[\hdim \{y\in[0,1]:\|nx-y\|<n^{-\kappa}\text{ for infinitely many }n\}=\frac{1}{\kappa},\]
	where $ \hdim $ stands for the Hausdorff dimension. The corresponding uniform approximation problem was recently studied by Kim and Liao~\cite{KimLi19} who obtained the Hausdorff dimension of the set
	\[\{y\in[0,1]:\forall N\gg1,~\exists n\le N, \text{ such that }\|nx-y\|<N^{-\kappa}\}.\]
	
	Naturally, one wonders about the analog results when the orbit $ (\{nx\})_{n\ge 0} $ is replaced by an orbit $ (T^nx)_{n\ge 0} $ of a general dynamical system $ ([0,1], T) $. For any $ \kappa>0 $, Fan, Schmeling and Troubetzkoy~\cite{FaScTr13} considered the set
	\[\mathcal L^{\kappa}(x):=\{y\in[0,1]:|T^nx-y|<n^{-\kappa}\text{ for infinitely many }n\}\]
	of points that are asymptotically approximated by the orbit $ (T^nx)_{n\ge 0} $ with a given speed $ n^{-\kappa} $, where $ T $ is the doubling map. It seems difficult to investigate the size of $ \mathcal L^{\kappa}(x) $ when $ x $ is not a dyadic rational, as the distribution of $ (T^nx)_{n\ge 0} $ is not as well-studied as that of $ (\{nx\})_{n\ge 0} $, see for example~\cite{AlBe98} for more details about the distribution of $ (\{nx\})_{n\ge 0} $.
	However, from the viewpoint of ergodic theory, Fan, Schmeling and Troubetzkoy~\cite{FaScTr13} obtained the Hausdorff dimension of $ \mathcal L^{\kappa}(x) $ for $ \uph $-a.e.\,$ x $, where $ \uph $ is the Gibbs measure associated with a H\"older continuous potential $ \phi $. They found that the size of $ \mathcal L^{\kappa}(x) $ is closely related to the local dimension of $ \uph $ and to the first hitting time for shrinking targets. 
	
	In their paper~\cite{LiaoSe13}, Liao and Seuret extended the results of~\cite{FaScTr13} to Markov maps. Later, Persson and Rams~\cite{PerRams17} considered more general piecewise expanding interval maps, and proved some similar results to those of~\cite{FaScTr13,LiaoSe13}. These studies are also closed related to the metric theory of random covering set; see~\cite{BaFa05, Dvo56,Fan02,JoSt08,Seu18,She72,Tan15} and references therein.
	
	
	As a counterpart of the dynamically defined asymptotic approximation set $ \mathcal L^{\kappa}(x) $, we would like to study the corresponding uniform approximation set $ \uk $ defined as
	\[\begin{split}
		\uk:&=\{y\in[0,1]:\forall N\gg1,~\exists n\le N, \text{ such that }|T^nx-y|<N^{-\kappa}\}\\
		&=\bigcup_{i=1}^\infty\bigcap_{N= i}^\infty\bigcup_{n=1}^N B(T^nx, N^{-\kappa}),
	\end{split}\]
	where $ B(x,r) $ is the open ball of center $ x $ and radius $ r $, and $ T $ is a Markov map (see Definition~\ref{d:Markov map} below). 
	
	As the studies on $ \mathcal L^\kappa(x) $, we are interested in the sizes (Lebesgue measure and Hausdorff dimension) of $ \uk $. By a simple argument, one can check that $ \uk\setminus\{T^nx\}_{n\ge 0}\subset \mathcal L^\kappa(x) $. Thus, trivially one has $ \lambda(\uk) \le \lambda(\mathcal L^\kappa(x))$  and $ \hdim \uk\le\hdim\mathcal L^\kappa(x) $. Here, $ \lambda $ denotes the Lebesgue measure on $ [0,1] $.

	 Our first result asserts that for any $ \kappa>0 $, the Lebesgue measure and the Hausdorff dimension of $ \uk $ are constant almost surely with respect to a $ T $-invariant ergodic measure.

	\begin{thm}\label{t:sub}
		Let $ T $ be a Markov map on $ [0,1] $ and $ \nu $ be a $ T $-invariant ergodic measure. Then for any $ \kappa>0 $, both $ \lambda(\uk) $ and $ \hdim\uk $ are constants almost surely. 
	\end{thm}
	To further describe the size of $\uk$ for all most all points, we impose a stronger condition, the same as that of Fan, Schmeling and Troubetzkoy~\cite{FaScTr13} and Liao and Seuret~\cite{LiaoSe13}, that $ \nu $ is a Gibbs measure. Precisely, let $\phi$ be a H\"older continuous potential  and $ \uph $ be the associated Gibbs meaure. Let $ \bk:=[0,1]\setminus \uk $. We remark that the sets $\uk$ are decreasing (the sets $\bk$ are increasing) with respect to $\kappa$. Then, we want to ask, for $\mu_\phi$ almost all points $x$, how the size of $\uk$ (and $\bk$) changes with respect to $\kappa$. We thus would like to answer the following questions.
	\begin{enumerate}[(Q1)]
		\item When $\uk=[0,1]$ for $\mu_\phi$-a.e.\,$x$?  What is the critical value
		\[
		\kappa_{\phi}:=\sup\{\kappa\geq 0: \uk=[0,1] \ \text{for $\mu_\phi$-a.e.} \,x\}?
		\]
		\item When $\lambda(\uk)=1$ for $\mu_\phi$-a.e.\,$x$? What is the critical value
		\[
		\kappa_{\phi}^\lambda:=\sup\{\kappa \geq 0: \lambda(\uk)=1 \ \text{for $\mu_\phi$-a.e.} \,x\}?
		\]
		\item What are the Hausdorff dimensions of $\uk$ and $\bk$ for $ \uph $-a.e.\,$ x $? What are the critical value
	        \[
		\kappa_{\phi}^H:=\sup\{\kappa \geq 0: \hdim(\uk)=0 \ \text{for $\mu_\phi$-a.e.} \,x\}?
		\]
		
	\end{enumerate}

In this paper, we answer these questions when $ T $ is an expanding Markov map of the interval $ [0,1] $ with a finite partition---a Markov map, for short. 

\begin{defn}[Markov map]\label{d:Markov map}
	A transformation $ T:[0,1]\to[0,1] $ is an expanding Markov map with a finite partition provided that there is a partition of $ [0,1] $ into subintervals  $ I(i)=(a_i, a_{i+1}) $ for $ i=0,\dots,Q-1 $ with endpoints $ 0=a_0<a_1<\cdots<a_Q=1 $ satisfying the following properties.
	\begin{enumerate}[\upshape(1)]
		\item There is an integer $ n_0 $ and a real number $ \rho $ such that $ |(T^{n_0})'|\ge\rho>1 $.
		\item $ T $ is strictly monotonic and can be extended to a $ C^2 $ function on each $ \overline{I(i)} $.
		\item If $ I(j)\cap T\big(I(k)\big)\ne\emptyset $, then $ I(j)\subset T\big(I(k)\big) $.
		\item There is an integer $ R $ such that $ I(j)\subset \cup_{n=1}^R T^n\big(I(k)\big) $ for every $ k $, $ j $.
		\item  For every $ k\in\{0,1,\dots,Q-1\} $, $ \sup_{(x,y,z)\in I(k)^3}\big(|T''(x)|/|T'(y)||T'(z)|\big)<\infty $.
	\end{enumerate}
\end{defn}

	For a probability measure $ \nu $ on $[0,1]$ and a point $ y\in[0,1] $, we set
	\[\ud_\nu(y):=\liminf_{r\to 0}\frac{\log\nu\big(B(y,r)\big)}{\log r}\quad\text{and}\quad\od_\nu(y):=\limsup_{r\to 0}\frac{\log\nu\big(B(y,r)\big)}{\log r},\]
	which are called, respectively, the lower and upper local dimensions of $ \nu $ at $y$. When $ \ud_\nu(y)=\od_\nu(y) $, their common value is denoted by $ d_\nu(y) $, and is simply called the local dimension of $ \nu $ at $ y $. 
Let $ D_\nu $ be the multifractal spectrum of $ \nu $ defined by
\[D_\nu(s):= \hdim\{y\in[0,1]:d_\nu(y)=s\} \quad \forall s\in\mathbb{R}.\]

Let $ T $ be a Markov map as in Definition \ref{d:Markov map} and let $ \phi $ be a  H\"older continuous potential. Denote by $ \mi $ the set of $ T $-invariant probability measures on $ [0,1] $ and $ \mu_{\max} $ the Gibbs measure associated with the potential $ -\log|T'| $. Define 
\begin{align}
	\alpha_-:&=\min_{\nu\in\mi}\frac{-\int\phi \,d\nu}{\int\log|T'|\,d\nu},\label{eq:alpha-}\\ 
	\alpha_{\max}:&=\frac{-\int\phi \, d\mu_{\max}}{\int\log|T'|\, d\mu_{\max}},\label{eq:alphamax}\\
	\alpha_+:&=\max_{\nu\in\mi}\frac{-\int\phi \, d\nu}{\int\log|T'|\, d\nu}.\label{eq:alpha+}
\end{align}
 By definition, it holds that $ \alpha_-\le \alpha_{\max}\le\alpha_+ $. Indeed, the quantities $ \alpha_- $, $ \alpha_{\max} $ and $ \alpha_+ $ depend on $ T $ and $ \phi $. However, for simplicity, we leave out the dependence unless the context requires specification. 

The following main theorem tells us that the three critical values demanded in questions (Q1)-(Q3) are $ \alpha_+ $, $ \alpha_{\max} $ and $ \alpha_-$, correspondingly. 

\begin{thm}\label{t:main}
	Let $ T $ be a Markov map. Let $ \phi $ be a H\"older continuous potential and $ \uph $ be the associated Gibbs measure.
	\begin{enumerate}[\upshape(1)]
		\item The critical value $ \kappa_{\phi} $ is $ \alpha_+ $. Namely, for $ \uph $-a.e.\,$ x $,
		\ $ \uk=[0,1] $ if $ 1/\kappa>\alpha_+ $, and $ \uk\ne[0,1] $ if $ 1/\kappa<\alpha_+ $.
		\item The critical value $ \kappa_{\phi}^\lambda $ is $ \alpha_{\max} $. Moreover, for $ \mu_\phi $-a.e.\,$ x $,
		\[\lambda\big(\uk\big)=1-\lambda\big(\bk\big)=\begin{cases}
			0\quad&\text{if }1/\kappa\in(0,\alpha_{\max}),\\
			1\quad&\text{if }1/\kappa\in(\alpha_{\max},+\infty).
		\end{cases}\]
		\item The critical value $ \kappa_{\phi}^H $ is $ \alpha_{-} $. Moreover, for $ \mu_\phi $-a.e.\,$ x $,
		\[\hdim\uk=\begin{cases}
			D_{\mu_\phi}(1/\kappa) &\text{if }1/\kappa\in(0,\alpha_{\max}]\setminus\{\alpha_-\},\\
			1&\text{if }1/\kappa\in(\alpha_{\max},+\infty).
		\end{cases}\]
	\item For $ \mu_\phi $-a.e.\,$ x $,
	\[\hdim\bk=\begin{cases}
		1&\text{if }1/\kappa\in(0,\alpha_{\max}),\\
		D_{\mu_\phi}(1/\kappa) &\text{if }1/\kappa\in[\alpha_{\max},+\infty)\setminus \{\alpha_+\}.
	\end{cases}\]
	\end{enumerate}
\end{thm}

\begin{rem}
	It is worth noting that the multifractal spectrum $ D_{\uph}(s) $ vanishes if $ s\notin[\alpha_-,\alpha_+] $. So if $ 1/\kappa<\alpha_- $, then $ \hdim\uk=0 $ for $ \uph $-a.e.\,$ x $.
\end{rem}

\begin{rem}
	The cases $ 1/\kappa=\alpha_- $ and $ \alpha_+ $ are not covered by Theorem \ref{t:main}. However, if the multifractal spectrum $ D_{\uph} $ is continuous at $ \alpha_- $ (respectively $ \alpha_+ $), we get that $ \hdim\mathcal U^{\alpha_-}(x)=0 $ (respectively $ \hdim\mathcal B^{\alpha_+}(x)=0 $) for $ \uph $-a.e.\,$ x $. The situation becomes more subtle if $ D_{\uph}(\cdot) $ is discontinuous at $ \alpha_- $ (respectively $ \alpha_+ $). Our methods do not work for obtaining the value of $ \hdim\mathcal U^{\alpha_-}(x) $ (respectively $ \hdim\mathcal B^{\alpha_+}(x)$) for $ \uph $-a.e.\,$ x $.
\end{rem}

Let $ \hdim \nu $ be the dimension of the Borel probability measure $ \nu $ defined by
\[\hdim \nu=\inf\{\hdim E:E\text{ is Borel set of }[0,1]\text{ and }\nu(E)>0\}.\]
\begin{rem}
	As already discussed above, $ \uk\setminus\{T^nx\}_{n\ge 0}\subset \mathcal L^\kappa(x) $, one may wonder whether the sets $ \uk $ and $ \mathcal L^\kappa(x) $ are essentially different. More precisely, is it possible that $ \hdim\uk $ is strictly less than $ \hdim\mathcal L^\kappa(x) $? Theorem~\ref{t:main} affirmatively answers this question. Compared with the asymptotic approximation set $ \mathcal L^{\kappa}(x) $, the structure of the uniform approximation set $ \uk $ does have a notable feature. When $ 1/\kappa\in(0,\hdim\uph)\setminus\{\alpha_-\} $, the map $ 1/\kappa\mapsto\hdim\uk $ agrees with the multifractal spectrum $ D_{\uph}(1/\kappa) $, while the map $ 1/\kappa\mapsto\hdim\mathcal L^{\kappa}(x) $ is the linear function $ f(1/\kappa)=1/\kappa $ independent of the multifractal spectrum. Therefore, $ \hdim\uk<\hdim\mathcal L^\kappa(x) $. See Figure 1 for an illustration.
\end{rem}
\begin{figure}[H]
	\centering
	 \includegraphics[height=10cm, width=15cm]{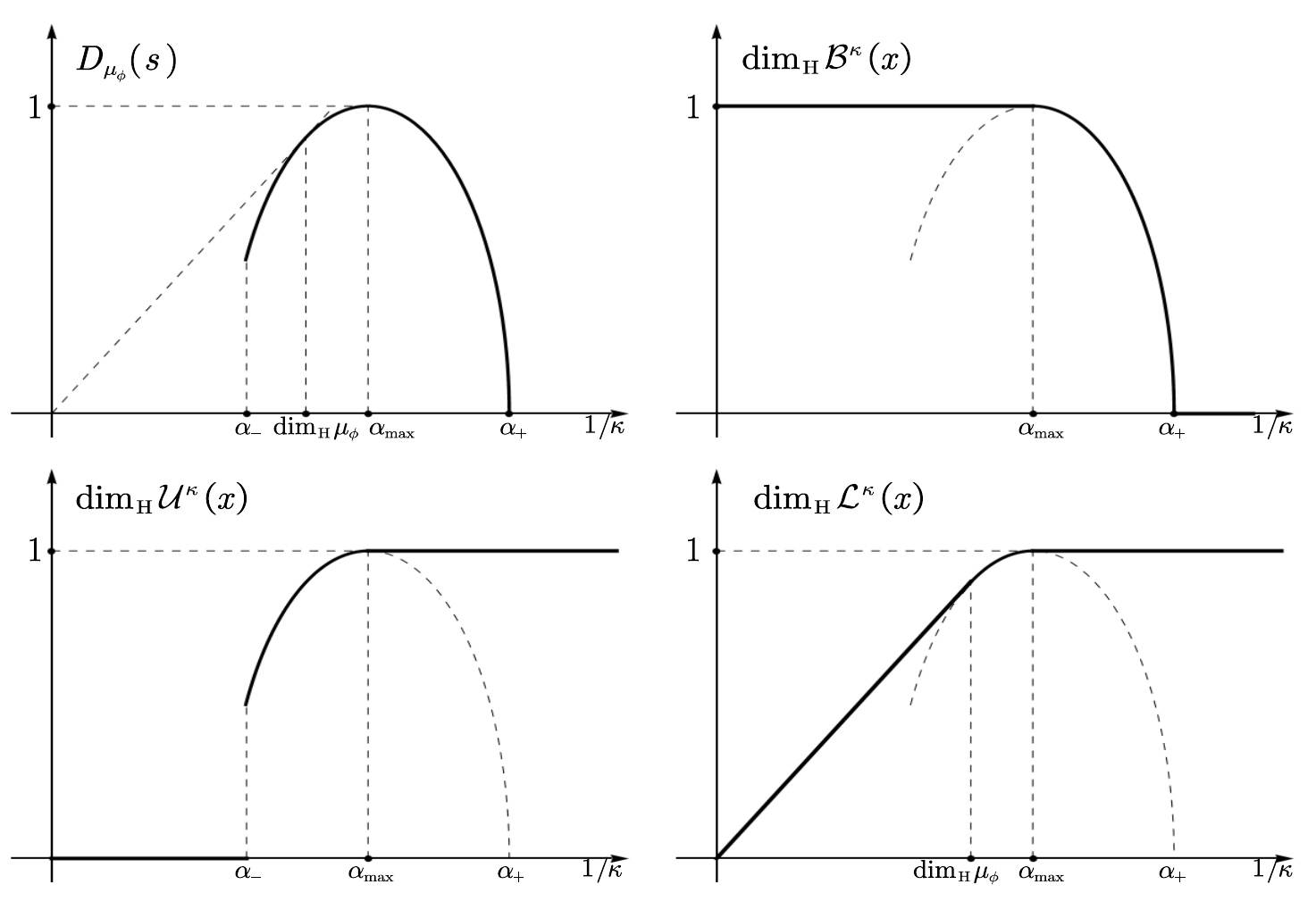} 
	 \caption*{\footnotesize Figure 1: The multifractal spectrum of $ \uph $ and the maps $ 1/\kappa\mapsto\hdim\bk $,  $ 1/\kappa\mapsto\hdim\uk $ and $ 1/\kappa\mapsto\hdim\mathcal L^\kappa(x) $.} 
\end{figure}

\begin{rem}
For the asymptotic approximation set $ \mathcal L^\kappa(x) $, the most difficult part lies in establishing the lower bound for $ \hdim\mathcal L^\kappa(x)$ when $ 1/\kappa<\hdim\uph $, for which a multifractal mass transference principle for Gibbs measure is applied, see \cite[\S 8]{FaScTr13}, \cite[\S 5.2]{LiaoSe13} and~\cite[\S 6]{PerRams17}. Specifically, since $ \uph(\mathcal L^{\delta}(x))=1 $ for all $ 1/\delta>\hdim\uph $, the multifractal mass transference principle guarantees the lower bound $ \hdim \mathcal L^{\kappa}(x)\ge (\hdim\uph)\delta/\kappa $ for all $ 1/\kappa<\hdim\uph $. By letting $ 1/\delta $ monotonically decrease to $ \hdim\uph $ along a sequence $ (\delta_n) $, we get immediately the expected lower bound $ \hdim \mathcal L^{\kappa}(x)\ge 1/\kappa $ for all $ 1/\kappa<\hdim\uph $. However, recent progresses in uniform approximation~\cite{BuLi16,KimLi19,KoLiPe21,ZhWu20} indicate that there is no mass transference principle for uniform approximation set. Therefore, we can not expect that $ \hdim\uk $ decreases linearly with respect to $ 1/\kappa $ as $ \hdim\mathcal L^\kappa(x) $ does. The main new ingredient of this paper is the difficult upper bound for $ \hdim\uk $ when $ 1/\kappa<\hdim\uph $. To overcome the difficulty, we fully develop and combine the methods in~\cite{FaScTr13} and~\cite{KoLiPe21}.
\end{rem}

To illustrate our main theorem, let us give some examples.
\begin{exmp}\label{ex:example 1}
	Suppose that $ T $ is the doubling map and $ \uph:=\lambda $ is the Lebesgue measure. Applying Theorem~\ref{t:main}, we have that for $ \lambda $-a.e.\,$ x $,
	\[\hdim\uk=\begin{cases}
		0 &\text{if }1/\kappa\in(0,1),\\
		1&\text{if }1/\kappa\in(1,+\infty).
	\end{cases}\]
	The Lebesgue measure is monofractal and hence the corresponding multifractal spectrum $ D_\lambda $ is discontinuous at $ 1 $. Theorem~\ref{t:main} fails to provide any metric statement for the set $ \mathcal U^1(x) $ for $ \lambda $-a.e.\,$ x $. However, we can conclude that $ \mathcal U^1(x) $ is a Lebesgue null set for $ \lambda $-a.e.\,$ x $ from Fubini Theorem and a zero-one law established in~\cite[Theorem 2.1]{HKKP21}. Further, by Theorem~\ref{t:sub}, $ \hdim\mathcal U^1(x) $ is constant for $ \lambda $-a.e.\,$ x $.  
\end{exmp}
	Some sets similar to $ \mathcal U^1(x) $ in Example~\ref{ex:example 1} have recently been studied by Koivusalo, Liao and Persson~\cite{KoLiPe21}. In their paper, instead of the orbit $ (T^nx)_{n\ge 0} $, they investigated the sets of points uniformly approximated by an independent and identically distributed sequence $ (\omega_n)_{n\ge 1} $. Specifically, they showed that with probability one, the lower bound of the Hausdorff dimension of the set
	\[\begin{split}
		\{y\in[0,1]:\forall N\gg1,~\exists n\le N, \text{ such that }|\omega_n-y|<1/N\}
	\end{split}\]
	is larger than $ 0.2177444298485995 $ (\cite[Theorem 5]{KoLiPe21}).

\begin{exmp}
	Let $ p\in(1/2,1) $. Suppose that $ T $ is the doubling map on $ [0,1] $ and $ \mu_p $ is the $ (p,1-p) $ Bernoulli measure. It is known that the multifractal spectrum $ D_{\mu_p} $ is continuous on $ (0,+\infty) $ and attains its unique maximal value $ 1 $ at $ -\log_2 \big(p(1-p)\big)/2 $. Theorem~\ref{t:main} then gives that for $ \mu_p $-a.e.\,$ x $,
	\[\hdim\uk=\begin{dcases}
		D_{\mu_p}(1/\kappa) &\text{if }1/\kappa\in\bigg(0,\frac{-\log_2 \big(p(1-p)\big)}{2}\bigg),\\
		1&\text{if }1/\kappa\in\bigg[\frac{-\log_2 \big(p(1-p)\big)}{2},+\infty\bigg).
	\end{dcases}\]
\end{exmp}

Our paper is organized as follows. We start in Section 2 with some preparations on Markov map, and then apply ergodic theory to give a proof of Theorem~\ref{t:sub}. Section 3 contains some recalls on multifractal
analysis and a variational principle which are essential in the proof of
Theorem~\ref{t:main}. Section 4 describes some relations among hitting time, approximating rate and local dimension of $ \uph $. From these relations we then derive items (1), (2) and (4) of Theorem~\ref{t:main} in Section 5.1, as well as the lower bound of $ \hdim\uk $ in Section 5.2. In the same Section 5.2, we establish the upper bound of $ \hdim\uk $, which is arguably the most substantial part.
\section{Basic definitions and the proof of Theorem~\ref{t:sub}}
\subsection{Covering of $ [0,1] $ by basic intervals}
Let $ T $ be a Markov map as defined in Definition~\ref{d:Markov map}. For each $ (i_1i_2\cdots i_n)\in\{0,1,\dots,Q-1\}^n $, we call
	\[I(i_1i_2\cdots i_n):=I(i_1)\cap T^{-1}\big(I(i_2)\big)\cap\dots\cap T^{-n+1}\big(I(i_n)\big)\]
	a basic interval of generation $ n $.  It is possible that $ I(i_1i_2\cdots i_n) $ is empty for some $ (i_1i_2\cdots i_n)\in\{0,1,\dots,Q-1\}^n $.
	The collection of non-empty basic intervals of a given generation $ n $ will be denoted by $ \Sigma_n $. For any $ x\in[0,1] $, we denote $ I_n(x) $ the unique basic interval $ I\in\Sigma_n $ containing $ x $.

	By the definition of Markov map, we obtain the following bounded distortion property on basic intervals: there is a constant $ L>1 $ such that for any $ x\in[0,1] $,
	\begin{equation}\label{eq:bdp}
		\text{for any }n\ge 1,\quad L^{-1}|(T^n)'(x)|^{-1}\le |I_n(x)|\le L|(T^n)'(x)|^{-1},
	\end{equation}
	where $ |I| $ is the length of the interval $ I $. 
Consequently,
we can find two constants $ 1<L_1<L_2 $ such that 
\begin{equation}\label{eq:length of basic interval}
	\text{for every } I\in \Sigma_n,\quad L_2^{-n}\le |I|\le L_1^{-n}.
\end{equation}

\subsection{Proof of Theorem~\ref{t:sub}}
Let us start with a simple but crucial observation.
\begin{lem}\label{l:invariant}
	Let $ T $ be a Markov map. For any $ \kappa>0 $, we have $ \uk\setminus\{Tx\}\subset \mathcal U^{\kappa}(Tx) $.
\end{lem}
\begin{proof}
	Let $ y\in\uk\setminus\{Tx\} $ and let $ i $ be the smallest integer satisfying $ y\notin B(Tx, 2^{-i\kappa}) $. By the definition of $ \uk $, for any integer $ N> 2^i $ large enough, there exists $ 1\le n\le N $ such that $ y\in B(T^nx, N^{-\kappa}) $. Moreover, the condition $ N> 2^i $ implies that $ n\ne1 $, hence $ y\in B\big(T^{n-1}(Tx),N^{-\kappa}\big) $. This gives that $ y\in\mathcal U^\kappa(Tx) $. Therefore, $ \uk\setminus\{Tx\}\subset \mathcal U^{\kappa}(Tx) $.
\end{proof}
Recall that a $ T $-invariant measure $ \nu $ is ergodic if and only if any $ T $-invariant function is constant almost surely. The proof of Theorem~\ref{t:sub} falls naturally into two parts. We first deal with the Lebesgue measure part.

\begin{proof}[Proof of Theorem~\ref{t:sub}: Lebesgue measure part]
	For any $ \kappa>0 $, define the function $ g_\kappa:x\mapsto\lambda(\uk) $. We claim that $ g_\kappa $ is measurable. In fact, it suffices to observe that
	\[g_\kappa(x)=\lim_{i\to \infty}\lim_{m\to\infty}\lambda\bigg(\bigcap_{N=i}^m\bigcup_{n=1}^NB(T^nx,N^{-\kappa})\bigg)\]
	and that, by the piecewise continuity of $ T^n $ ($ n\ge 1 $), the set
	\[\bigg\{x\in[0,1]:\lambda\bigg(\bigcap_{N=i}^m\bigcup_{n=1}^NB(T^nx,N^{-\kappa})\bigg)>t\bigg\}\]
	is measurable for any $ t\in\R $.
	
	By Lemma~\ref{l:invariant}, we see that $ \lambda(\mathcal U^\kappa(Tx))\ge \lambda(\uk) $, or equivalently, $ g_\kappa(Tx)\ge g_\kappa(x) $. Since $ g_\kappa $ is measurable, by the fact $ 0\le g_\kappa\le 1 $ and the invariance of $ \nu $, we have that $ g_\kappa $ is invariant with repect to $ \nu $, that is, $ g_\kappa(Tx)=g_\kappa(x) $ for $ \nu $-a.e.\,$ x $. In the presence of ergodicity of $ \nu $, $ g_\kappa $ is constant almost surely.
\end{proof}

\begin{proof}[Proof of Theorem~\ref{t:sub}: Hausdorff dimension part]

	Fix $ \kappa>0 $, define the function $ f_\kappa:x\mapsto\hdim\uk $. Again by Lemma~\ref{l:invariant}, we see that $ f_\kappa(Tx)\ge f_\kappa(x) $. Proceed in the same way as in the Lebesgue measure part yields that $ f_\kappa $ is constant almost surely provided that $ f_\kappa $ is measurable.
	
	To show that $ f_\kappa $ is measurable, it suffices to prove that for any $ t>0 $, the set
	\[\begin{split}
		A(t):=\{x\in[0,1]:f_\kappa(x)<t\}=\bigg\{x\in[0,1]:\hdim\bigg(\bigcup_{i=1}^\infty\bigcap_{N= i}^\infty\bigcup_{n=1}^NB(T^nx, N^{-\kappa})\bigg)<t\bigg\}
	\end{split}\]
	is measurable. 
	Through out the proof of this part, we will assume that the ball $ B(T^nx, N^{-\kappa}) $ is closed. This makes the proof achievable, and it does not change the Hausdorff dimension of $ \uk $.
	
	By the definition of Hausdorff dimension, a point $ x\in A(t) $ if and only if there exists $ h\in \N $ such that 
	\[\mathcal H^{t-\frac{1}{h}}\bigg(\bigcup_{i=1}^\infty\bigcap_{N= i}^\infty\bigcup_{n=1}^NB(T^nx, N^{-\kappa})\bigg)=0,\]
	or equivalently for all $ i\ge 1 $,
	\begin{equation}\label{eq:ht-1/j=0}
		\mathcal H^{t-\frac{1}{h}}\bigg(\bigcap_{N= i}^\infty\bigcup_{n=1}^NB(T^nx, N^{-\kappa})\bigg)=0.
	\end{equation}
	By the definition of Hausdorff measure,~\eqref{eq:ht-1/j=0} holds if and only if for any $ j,k\in\N $, 
	\[\mathcal H_{\frac{1}{j}}^{t-\frac{1}{h}}\bigg(\bigcap_{N= i}^\infty\bigcup_{n=1}^NB(T^nx, N^{-\kappa})\bigg)<\frac{1}{k}.\]
	Hence, we see that
%
	\[\begin{split}
		A(t)=\bigcup_{h=1}^\infty\bigcap_{i=1}^\infty\bigcap_{j=1}^\infty\bigcap_{k=1}^\infty\bigg\{ x\in [0,1]:\mathcal H_{\frac{1}{j}}^{t-\frac{1}{h}}\bigg(\bigcap_{N= i}^\infty\bigcup_{n=1}^NB(T^nx, N^{-\kappa})\bigg)<\frac{1}{k}\bigg\}=:\bigcup_{h=1}^\infty\bigcap_{i=1}^\infty\bigcap_{j=1}^\infty\bigcap_{k=1}^\infty B_{h,i,j,k}.
	\end{split}\]


	If  $ x\in B_{h,i,j,k} $, then there is a countable open cover $ \{U_p\}_{p\ge 1} $ with $ 0<|U_p|<1/j $ satisfying
	\begin{equation}\label{eq:subset Up}
		\bigcap_{N= i}^\infty\bigcup_{n=1}^NB(T^nx, N^{-\kappa})\subset \bigcup_{p=1}^\infty U_p\quad \text{and}\quad\sum_{p=1}^\infty |U_p|^{t-\frac{1}{h}}<\frac{1}{k}.
	\end{equation}
	The set $ \bigcap_{N= i}^\infty\bigcup_{n=1}^NB(T^nx, N^{-\kappa}) $ can be viewed as the intersection of a family of decreasing compact sets $ \big\{\bigcap_{N= i}^l\bigcup_{n=1}^NB(T^nx, N^{-\kappa})\big\}_{l\ge i} $. Hence there exists $ l_0\ge i $ satisfying
	\[\bigcap_{N= i}^{l_0}\bigcup_{n=1}^NB(T^nx, N^{-\kappa})\subset\bigcup_{p=1}^\infty U_p,\]
	which implies
	\[\mathcal H_{\frac{1}{j}}^{t-\frac{1}{h}}\bigg(\bigcap_{N= i}^{l_0}\bigcup_{n=1}^NB(T^nx, N^{-\kappa})\bigg)<\frac1k.\]
	We then deduce that
	\begin{equation}\label{eq:Bijkl subset Cijklm}
		B_{h,i,j,k}\subset\bigcup_{l=i}^{\infty}\bigg\{ x\in [0,1]:\mathcal H_{\frac{1}{j}}^{t-\frac{1}{h}}\bigg(\bigcap_{N= i}^l\bigcup_{n=1}^NB(T^nx, N^{-\kappa})\bigg)<\frac1k\bigg\}=:\bigcup_{l=i}^{\infty}C_{h,i,j,k,l}.
	\end{equation}

	If $ x\in C_{h,i,j,k,l} $ for some $ l\ge 1 $, then
	\[\mathcal H_{\frac{1}{j}}^{t-\frac{1}{h}}\bigg(\bigcap_{N= i}^\infty\bigcup_{n=1}^NB(T^nx, N^{-\kappa})\bigg)\le\mathcal H_{\frac{1}{j}}^{t-\frac{1}{h}}\bigg(\bigcap_{N= i}^l\bigcup_{n=1}^NB(T^nx, N^{-\kappa})\bigg)<\frac{1}{k}.\]
	Hence $ x\in B_{i,j,k,l} $ and the reverse inclusion of~\eqref{eq:Bijkl subset Cijklm} is proved.
	
	Notice that $ T,T^2,\dots,T^l $ are continuous on every basic interval of generation greater than $ l $. For any $ x\in C_{h,i,j,k,l} $, denote $$ \mathcal S(x):=\bigcap_{N= i}^l\bigcup_{n=1}^NB(T^nx, N^{-\kappa}). $$ There is an open cover $ (V_p)_{p\ge 1} $ of $ \mathcal S(x) $ with $ 0<|V_p|<1/j $ and $ \sum_p |V_p|^{t-1/h}<1/k $. Since $ \mathcal S(x) $ is compact, we see that the distance $ \delta $ between $ \mathcal S(x) $ and the complement of $ \bigcup_{p\ge 1} V_p $ is positive. By the continuities of $ T,T^2,\dots,T^l $, there is some $ l_1:=l_1(\delta)>l $, for which if $ y\in I_{l_1}(x) $, then $ \mathcal S(y) $ is contained in a $ \delta/2 $-neighborhood of $ \mathcal S(x) $. Thus $ \mathcal S(y) $ can also be covered by $ \bigcup_{p\ge 1} V_p $, and $ I_{l_1}(x)\subset C_{h,i,j,k,l} $. Finally, $ C_{h,i,j,k,l} $ is a union of some basic intervals, which is measurable.
	
	Now combining the equalities obtained above, we have
	\[A(t)=\bigcup_{h=1}^\infty\bigcap_{i=1}^\infty\bigcap_{j=1}^\infty\bigcap_{k=1}^\infty\bigcup_{l=i}^{\infty}C_{h,i,j,k,l},\]
	which is a Borel measurable set.
\end{proof}


\section{Multifractal properties of Gibbs measures}\label{s:Multifractal properties}
In this section, we review some standard facts on multifractal properties of Gibbs measures.
\begin{defn}\label{d:Gibbs measure}
	A Gibbs measure $ \uph $ associated with a potential $ \phi $ is a probability measure satisfying the following: there exists a constant $ \gamma>0 $ such that
	\[\text{for any basic interval }I\in\Sigma_n, \quad\gamma^{-1}\le\frac{\uph(I)}{e^{S_n\phi(x)-nP(\phi)}}\le\gamma,\quad\text{for every }x\in I,\]
	where $ S_n\phi(x)=\phi(x)+\cdots+\phi(T^{n-1}x) $ is the $ n $th Birkhoff sum of $ \phi $ at $ x $, and $ P(\phi) $ is the topological pressure of $\phi$ defined by
	\[P(\phi)=\lim_{n\to\infty}\frac{1}{n}\log\sum_{I\in\Sigma_n}\sup_{x\in I}e^{S_n\phi(x)}.\]
\end{defn}
The following theorem ensures the existence and uniqueness of invariant Gibbs measure.

\begin{thm}[\cite{Bal00,Wal78}]
	Let $ T:[0,1]\to[0,1] $ be a Markov map. Then for any H\"older continuous function $ \phi $, there exists a unique $ T $-invariant Gibbs measure $ \uph $ associated with $ \phi $. Further, $ \uph $ is ergodic.
\end{thm}
The Gibbs measure $ \uph $ also satisfies the quasi-Bernoulli property (see~\cite[Lemma 4.1]{LiaoSe13}), i.e. for any $ n>k\ge 1 $, for any basic interval $ I(i_1\cdots i_n)\in\Sigma_{n} $, the following holds
\begin{equation}\label{eq:quasi-Bernoulli}
	\gamma^{-3}\uph(I')\uph(I'')\le\uph(I)=\uph(I'\cap T^{-k}I'')\le \gamma^3\uph(I')\uph(I''),
\end{equation}
where $ I'=I(i_1\cdots i_k)\in\Sigma_k $ and $ I''=I(i_{k+1}\cdots i_n)\in\Sigma_{n-k} $. It follows immediately that
\begin{equation}\label{eq:quasi-Bernoulli consequence}
	\text{for any }m\ge k,\quad\uph(I'\cap T^{-m}U)\le \gamma^3\uph(I')\uph(U),
\end{equation}
where $ U $ is an open set in $ [0,1] $. 

We adopt the convention that $ \phi $ is	 normalized, i.e. $ P(\phi)=0 $. If it is not the case, we can replace $\phi$ by $ \phi-P(\phi) $.

Now, let us recall some standard facts on multifractal analysis which aims at studying the multifractal spectrum $ D_{\uph} $. Some multifractal analysis results were summarised in~\cite{LiaoSe13} and we present them as follows. The proofs can also be found in the references~\cite{BaPeS97,BrMiP92, CoLeP87, PeWe97,Ran89, Sim94}. 
\begin{thm}[{\cite[Theorem 2.5]{LiaoSe13}}]\label{t:multifractal analysis}
	Let $ T $ be a Markov map. Let $ \phi $ be a H\"older continuous potential and $ \uph $ be the associated Gibbs measure. Then, the following hold.
	\begin{enumerate}[\upshape(1)]
		\item The function $ D_{\uph} $ of $ \uph $ is a concave real-analytic map on the interval $ (\alpha_-,\alpha_+) $, where $ \alpha_- $ and $ \alpha_+ $ are defined in~\eqref{eq:alpha-} and~\eqref{eq:alpha+}, respectively.
		\item The spectrum $ D_{\uph} $ reaches its maximum value $ 1 $ at $ \alpha_{\max} $ defined in~\eqref{eq:alphamax}.
		\item The graph of $ D_{\uph} $ and the first bisector intersect at a unique point which is $ (\hdim\uph,\hdim\uph) $. Moreover, $ \hdim\uph $ satisfies
		\[\hdim\uph=\frac{-\int\phi\, d\uph}{\int\log|T'|\, d\uph}.\]
	\end{enumerate}
\end{thm}
\begin{prop}[{\cite[\S 2.3]{LiaoSe13}}]\label{p:topological pressure equals to 0}
	For every $ q\in\R $, there is a unique real number $ \eta_\phi(q) $ such that the topological pressure $ P\big(-\eta_\phi(q)\log |T'|+q\phi\big) $ equals to $ 0 $. Further, $ \eta_\phi(q) $ is real-analytic and concave. 
\end{prop}
\begin{rem}\label{r:Gibbs measure}
	For simplicity, we denote by $ \mu_q $
	the $ T $-invariant Gibbs measure associated with the potential $ -\eta_\phi(q)\log |T'|+q\phi $. Certainly, $ \eta_\phi(0)=1 $ and the corresponding measure $ \mu_0 $ is associated with the potential $ -\log|T'| $. By the bounded distortion property~\eqref{eq:bdp}, the Gibbs measure $ \mu_0 $, coinciding with $ \mu_{\max} $, is strongly equivalent to the Lebesgue measure $ \lambda $.
\end{rem}

For every $ q\in\R $, we introduce the exponent
\begin{equation}\label{eq:alpha(q)}
	\alpha(q)=\frac{-\int\phi\, d\mu_q}{\int\log|T'|\, d\mu_q}.
\end{equation}
\begin{prop}[{\cite[\S 2.3]{LiaoSe13}}]\label{p:proposition of muq and alpha(q)}
	Let $ \mu_q $ and $ \alpha(q) $ be as above. The following statements hold.
	\begin{enumerate}[\upshape(1)]
		\item The Gibbs measure $ \mu_q $ is supported by the level set $ \{y:d_{\uph}(y)=\alpha(q)\} $ and $ D_{\uph}\big(\alpha(q)\big)=\hdim\mu_q=\eta_\phi(q)+q\alpha(q) $.
		\item The map $ \alpha(q) $ is decreasing, and
		\begin{align*}
			\lim_{q\to+\infty}\alpha(q)=\alpha_-,&\quad\lim_{q\to-\infty}\alpha(q)=\alpha_+,\notag\\
			\alpha(1)=\hdim\uph,&\ \ \quad\alpha(0)=\alpha_{\max}.\label{eq:alpha(0)}
		\end{align*}
		\item The inverse of $ \alpha(q) $ exists, and is denoted by $ q(\alpha) $. Moreover, $ q(\alpha)<0 $ if $ \alpha\in (\alpha_{\max},\alpha_+) $, and $ q(\alpha)\ge 0 $ if $ \alpha\in (\alpha_-,\alpha_{\max}] $.
	\end{enumerate}
\end{prop}

For a probability measure $ \nu $ on $[0,1]$ and a point $ y\in[0,1] $, define the lower and upper Markov pointwise dimensions respectively by
\[\underline{M}_{\nu}(y):=\liminf_{n\to\infty}\frac{\log\nu\big(I_n(y)\big)}{\log |I_n(y)|},\quad \overline{M}_{\nu}(y):=\limsup_{n\to\infty}\frac{\log\nu\big(I_n(y)\big)}{\log |I_n(y)|}.\]
When $ \underline{M}_{\nu}(y)=\overline{M}_{\nu}(y) $, their common value is denoted by $ M_\nu(y) $. By~\eqref{eq:bdp} and~\eqref{eq:length of basic interval}, we have
\[\od_{\nu}(y)\le \overline{M}_{\nu}(y),\]
which implies the inclusions
\begin{equation}\label{eq:local<upper<Markov dimension}
	\{y:d_{\nu}(y)=s\}
	\subset\{y:\ud_{\nu}(y)\ge s\}\subset\{y:\od_{\nu}(y)\ge s\}\subset\{y:\overline{M}_{\nu}(y)\ge s\}.
\end{equation}

By the Gibbs property of $ \uph $ and the bounded distortion property on basic intervals~\eqref{eq:bdp}, the definitions of Markov pointwise dimensions can be reformulated as
\begin{equation}
	\overline{M}_{\uph}(y)=\limsup_{n\to \infty}\frac{S_n\phi(y)}{S_n(-\log T')(y)}\quad\text{and}\quad M_{\uph}(y)=\lim_{n\to \infty}\frac{S_n\phi(y)}{S_n(-\log T')(y)}.
\end{equation}
This allows us to derive the following lemma, which is an alternative version of a proposition due to Jenkinson~\cite[Proposition 2.1]{Jen06}. We omit its proof since the argument is similar.

\begin{lem}\label{l:>alpha+ empty}
	Let $ T $ be a Markov map. Let $ \phi $ be a H\"older continuous potential and let $ \uph $ be the corresponding Gibbs measure. Then,
	\[\sup_{y\in [0,1]}\overline{M}_{\uph}(y)=\sup_{y\colon M_{\uph}(y)\text{ exists}}M_{\uph}(y)=\max_{\nu\in \mi }\frac{-\int\phi \, d\nu}{\int\log|T'|\, d\nu}=\alpha_+.\]
	In particular, for any $ s>\alpha_+ $,
	\[\{y:d_{\uph}(y)=s\}=\{y:\od_{\uph}(y)\ge s\}=\emptyset.\]
\end{lem}

We finish the section with a variational principle.

\begin{lem}\label{l:dimension spectrum}
		Let $ T $ be a Markov map. Let $ \phi $ be a H\"older continuous potential and $ \uph $ be the associated Gibbs measure.
	\begin{enumerate}[\upshape(1)]
		\item For every $ s<\alpha_{\max} $, $$ \hdim\{y:\ud_{\uph}(y)\le s\}=\hdim\{y:\od_{\uph}(y)\le s\}=D_{\uph}(s). $$
		\item For every $ s\in(\alpha_{\max},+\infty)\setminus \alpha_+ $,
		$$ \hdim\{y:\ud_{\uph}(y)\ge s\}=\hdim\{y:\od_{\uph}(y)\ge s\}=D_{\uph}(s). $$
	\end{enumerate}
\end{lem}
\begin{proof}
	(1) We point out that the following inclusions hold
\begin{equation*}
	\{y:d_{\uph}(y)=s\}\subset\{y:\od_{\uph}(y)\le s\}\subset\{y:\ud_{\uph}(y)\le s\}.
\end{equation*}
	In~\cite[proposition 2.8]{LiaoSe13}, the leftmost set and the rightmost set were shown to have the same Hausdorff dimension. This together with the above inclusions completes the proof of the first point of the lemma.
	
	(2) When $ T $ is the doubling map, the statement was formulated by Fan, Schmeling and Trobetzkoy~\cite[Theorem 3.3]{FaScTr13}. Our proof follows their idea closely, we include it for completeness.
	
	By Lemma~\ref{l:>alpha+ empty}, we can assume without loss of generality that $ s<\alpha_+ $. The inclusions in~\eqref{eq:local<upper<Markov dimension} imply the following inequalities:
	\[\hdim\{y:d_{\uph}(y)=s\}
	\le \hdim\{y:\od_{\uph}(y)\ge s\}\le\hdim\{y:\overline{M}_{\uph}(y)\ge s\}.\]
	
	We turn to prove the reverse inequalites.
	By Proposition~\ref{p:proposition of muq and alpha(q)} and the condition $ s>\alpha_{\max} $, there exists a real number $ q_s:=q(s)<0 $ such that 
	\[s=\frac{-\int\phi\, d\mu_{q_s}}{\int\log|T'|\, d\mu_{q_s}} \quad\text{ and }\quad D_{\uph}(s)=\hdim\mu_{q_s}=\eta_\phi(q_s)+q_ss,\]
	where $ \mu_{q_s} $ is the Gibbs measure associated with the potential $ -\eta_\phi(q_s)\log|T'|+q_s\phi $. Now let $ y $ be any point such that $ \overline{M}_{\uph}(y)\ge s $. By Proposition~\ref{p:topological pressure equals to 0}, the topological pressure $ P(-\eta_\phi(q_s)\log|T'|+q_s\phi) $ is $ 0 $. Then we can apply the Gibbs property of $ \mu_{q_s} $ and~\eqref{eq:bdp} to yield 
	\[\begin{split}
		\underline{M}_{\mu_{q_s}}(y)&=\liminf_{n\to\infty}\frac{\log e^{S_n(-\eta_\phi(q_s)\log|T'|+q_s\phi)(y)}}{\log|I_n(y)|}\\
		&=\liminf_{n\to\infty}\bigg(\frac{-\eta_\phi(q_s)\log|(T^n)'(y)|}{\log|I_n(y)|}+q_s\cdot\frac{\log e^{S_n\phi(y)}}{\log|I_n(y)|}\bigg)\\
		&=\eta_\phi(q_s)+q_s\cdot\limsup_{n\to\infty}\frac{\log\uph \big(I_n(y)\big)}{\log|I_n(y)|}\\
		&=\eta_\phi(q_s)+q_s \overline{M}_{\uph}(y)\\
		&\le \eta_\phi(q_s)+q_ss=D_{\uph}(s),
	\end{split}\]
	where the inequality holds because $ q_s<0 $.
	
	Finally, Billingsley's Lemma~\cite[Lemma 1.4.1]{BiPe17} gives
	\[\hdim\{y:\overline{M}_{\uph}(y)\ge s\}\le\hdim\{y:\underline{M}_{\mu_{q_s}}(y)\le D_{\uph}(s)\}\le D_{\uph}(s)=\hdim\{y:d_{\uph}(y)=s\}.\qedhere \] 
	
\end{proof}


\section{Covering questions related to hitting time and local dimension}
In Section~\ref{ss:covering hitting time} below, we reformulate the uniform approximation set $ \uk $ in terms of hitting time. Thereafter, we relate the first hitting time for shrinking balls to the local dimensions in Section~\ref{ss:hitting time local dimension}.

\subsection{Covering questions and hitting time}\label{ss:covering hitting time}
\

Denote $ \mathcal O^+(x):=\{T^nx:n\ge 1\} $.
\begin{defn}
	For every $ x,y\in[0,1] $ and $ r>0 $, we define the first hitting time of the orbit of $ x $ into the ball $ B(y,r) $ by
	\[\tau_r(x,y):=\inf\{n\ge 1:T^nx\in B(y,r)\}.\]
\end{defn}
Set
\[\ur(x,y):=\liminf_{r\to 0}\frac{\log\tau_r(x,y)}{-\log r}\quad\text{and}\quad\ovr(x,y):=\limsup_{r\to 0}\frac{\log\tau_r(x,y)}{-\log r}.\]
For convenience, when $ \mathcal O^+(x)\cap B(y,r)=\emptyset $, we set $ \tau_r(x,y)=\infty $ and $ \ur(x,y)=\ovr(x,y)=\infty $. 
If $ \ur(x,y)=\ovr(x,y) $, we denote the common value by $ R(x,y) $. 

For any ball $ B\subset [0,1] $, we define the first hitting time $ \tau(x, B) $ by
\[\tau(x, B):=\inf \{n\ge 1:T^nx\in B\}.\]
Similarly, we set $ \tau(x, B)=\infty $ when $ \mathcal O^+(x)\cap B=\emptyset $.

The following lemma exhibits a relation between $ \uk $ and hitting time.
	\begin{lem}\label{l:described by hitting time}
		For any $ \kappa>0 $, we have
		\begin{align*}
			\bigg\{ y\in[0,1]:\ovr(x,y)>\frac{1}{\kappa} \bigg\}&\subset\bk\subset\bigg\{ y\in[0,1]:\ovr(x,y)\ge\frac{1}{\kappa} \bigg\},\\
			\bigg\{ y\in[0,1]:\ovr(x,y)<\frac{1}{\kappa} \bigg\}&\subset\uk\subset\bigg\{ y\in[0,1]:\ovr(x,y)\le\frac{1}{\kappa} \bigg\}.
		\end{align*}
	\end{lem}
	\begin{proof}
		The top left and bottom right inclusions imply one another. Let us prove the bottom right inclusion. Suppose that $ y\in \uk $. Then for all large enough $ N $ there is an $ n\le N $ such that $ T^nx \in B(y,N^{-\kappa})$. Thus $ \tau_{N^{-\kappa}}(x,y)\le N $ for all $ N $ large enough, which implies $ \ovr(x,y)\le 1/\kappa $.
		
		The top right and bottom left inclusions imply one another. So, it remains to prove the bottom left inclusion. Consider $ y $ such that $ \ovr(x,y)<1/\kappa $. If $ y\in\mathcal O^+(x)  $ with $ y=T^{n_0}x $ for some $ n_0\ge 1 $, then the system
		\[|T^nx-y|=|T^nx-T^{n_0}x|<N^{_\kappa}\quad\text{and}\quad 1\le n\le N\]
		always has a trivial solution $ n=n_0 $ for all $ N\ge n_0 $. Therefore $ y\in \uk $. Now assume that $ y\notin\mathcal O^+(x) $. By the definition of $ \ovr(x,y) $, there is a positive real number $ r_0<1 $ such that
		\[\tau_r(x,y)<r^{-1/\kappa}, \quad\text{for all }0<r<r_0.\]
		Denote $ n_r:=\tau_r(x,y) $, for all $ 0<r<r_0 $. Since $ y\notin \mathcal O^+(x) $, the family of positive integers $ \{n_r:0<r<r_0\} $ is unbounded. For each $ N> r_0^{-1/\kappa} $, denote $ t:=N^{-\kappa} $. The definition of $ n_{t} $ implies that
		\[T^{n_{t}}x\in B(y,t)=B(y,N^{-\kappa}).\]
		We conclude $ y\in\uk $ by noting that $ n_{t}<t^{-1/\kappa}=N $.
	\end{proof}

\subsection{Relation between hitting time and local dimension}\label{ss:hitting time local dimension}
As Lemma~\ref{l:described by hitting time} shows, we need to study the hitting time $ \ovr(x,y) $ of the Gibbs measure $ \uph $. We will prove that the hitting time is related to local dimension when the measure is exponential mixing.
\begin{defn}
	A $ T $-invariant measure $ \nu $ is exponential mixing if there exist two constants $ C>0 $ and $ 0<\beta<1 $ such that for any ball $ A $ and any Borel measurable set $ B $,
	\begin{equation}\label{eq:exponential mixing}
		|\nu(A\cap T^{-n}B)-\nu(A)\nu(B)|\le C\beta^n\nu(B).
	\end{equation}
\end{defn}
\begin{thm}[\cite{Bal00, LiSaV98,PaPo90, Rue04}]\label{t:exponential mixing}
	The $ T $-invariant Gibbs measure $ \uph $ associated with a H\"older continuous potential $ \phi $ of a Markov map $ T $ is exponential mixing. 
\end{thm}

The exponential mixing property allows us to apply the following theorem which decribes a relation between hitting time and local dimension of invariant measure. 

\begin{thm}[\cite{Gal07}]
	Let $ (X,T,\nu) $ be a measure-theoretic dynamical system. If $ \nu $ is superpolynomial mixing and if $ d_\nu(y) $ exists, then for $ \nu $-a.e.\,$ x $, we have
	\[R(x,y)=d_\nu(y).\]
\end{thm}

It should be noticed that superpolynomial mixing property is much weaker than exponential mixing property. 

Now, we turn to the study of the Markov map $ T $ on the interval $ [0,1] $. 
An application of Fubini's theorem yields the following corollary.
\begin{cor}[{\cite[Corollary 3.8]{LiaoSe13}}]\label{c:hitting time and local dimension}
		Let $ T $ be a Markov map. Let $ \uph $ and $ \mu_\psi $ be two $ T $-invariant Gibbs probability measures on $ [0,1] $ associated with H\"older potentials $ \phi $ and $ \psi $, respectively. Then,
	\[\text{for }\uph\times\mu_\psi\text{-a.e.\,}(x,y),\quad R(x,y)=d_{\uph}(y)=\frac{-\int\phi\, d\mu_\psi}{\int\log|T'|\, d\mu_\psi}.\]
\end{cor}

\section{Studies of $ \bk $ and $ \uk $} 
\subsection{The study of $ \bk $}
In this subsection, we are going to prove Theorem~\ref{t:main} except the
 item (3). Let us start with the lower bound for $ \hdim\bk $.
\begin{lem}\label{l:lower bound for hdimfk}
	Let $ T $ be a Markov map. Let $ \phi $ be a H\"older continuous potential and let $ \uph $ be the corresponding Gibbs measure. For any $ \kappa>0 $, the following hold.
	\begin{enumerate}[\upshape(a)]
		\item If $ 1/\kappa\in(0,\alpha_{\max}) $, then $ \lambda\big(\bk\big)=1 $ for $ \uph $-a.e.\,$ x $. 
		\item If $ 1/\kappa\in[\alpha_{\max},+\infty)\setminus\{\alpha_+\} $, then $ \hdim\bk\ge D_{\uph}(1/\kappa) $ for $ \uph $-a.e.\,$ x $. 
	\end{enumerate}
\end{lem}

\begin{proof}
	(a) Let $ 1/\kappa\in(0,\alpha_{\max}) $. As already observed in Section~\ref{s:Multifractal properties}, the Gibbs measure $ \mu_0 $ associated with $ \log|T'| $ is strongly equivalent to the Lebesgue measure $ \lambda $. Thus, a set $ F $ has full $ \mu_0 $-measure if and only if $ F $ has full $ \lambda $-measure. Corollary~\ref{c:hitting time and local dimension} implies that
	\[\text{for }\uph\times\mu_0\text{-}a.e.\,(x,y),\quad R(x,y)=d_{\uph}(y)=\frac{-\int\phi\, d\mu_0}{\int\log|T'|\, d\mu_0}=\alpha_{\max}.\]
	By Fubini's theorem, for $ \uph $-a.e.\,$ x $, the set $ \{y:R(x,y)=d_{\uph}(y)=\alpha_{\max}\} $ has full $ \mu_0 $-measure. Then for $ \uph $-a.e.\,$ x $, we have
	\[\mu_0\big(\{ y:\ovr(x,y)>1/\kappa \}\big)\ge \mu_0\big(\{y:R(x,y)=d_{\uph}(y)=\alpha_{\max}\}\big)=1.\]
	By Lemma~\ref{l:described by hitting time}, we arrive at the conclusion.
	
	(b) By Lemma~\ref{l:>alpha+ empty}, the level set $ \{y:d_{\uph}(y)=1/\kappa\} $ is empty if $ 1/\kappa>\alpha_+ $. Thus $ D_{\uph}(1/\kappa)=0 $, and therefore $ \hdim\bk\ge 0=D_{\uph}(1/\kappa) $ trivially holds for all $ 1/\kappa>\alpha_+ $. 
	
	Let $ 1/\kappa\in[\alpha_{\max},\alpha_+) $. We can suppose that $ \alpha_{\max}\ne \alpha_+ $, since otherwise $ [\alpha_{\max},\alpha_+)=\emptyset $, there is nothing to prove.	
	For any $ s\in (1/\kappa,\alpha_+) $, by Proposition~\ref{p:proposition of muq and alpha(q)}, there exists a real number $ q_s:=q(s) $ such that
	\[s=\frac{-\int\phi\, d\mu_{q_s}}{\int\log|T'|\, d\mu_{q_s}}\quad\text{and}\quad\mu_{q_s}(\{y:d_{\uph}(y)=s\})=1.\]
	Applying Corollary~\ref{c:hitting time and local dimension}, we obtain
	\[\text{for }\uph\times\mu_{q_s}\text{-}a.e.\,(x,y),\quad R(x,y)=d_{\uph}(y)=\frac{-\int\phi\, d\mu_{q_s}}{\int\log|T'|\, d\mu_{q_s}}=s.\]
	It follows from Fubini's theorem that, for $ \uph $-a.e.\,$ x $, the set $ \{y:R(x,y)=d_{\uph}(y)=s\} $ has full $ \mu_{q_s} $-measure. Consequently, for $ \uph $-a.e.\,$ x $,
	\[\begin{split}
		\hdim\{y:\ovr(x,y)>1/\kappa\}&\ge \hdim\{y:R(x,y)=d_{\uph}(y)=s\}\\
		&\ge \hdim \mu_{q_s}=D_{\uph}(s).
	\end{split}\]
We conclude by noting that $ s\in (1/\kappa,\alpha_+) $ is arbitrary and $ D_{\uph} $ is continuous on $ [\alpha_{\max},\alpha_+) $. 
\end{proof}

We are left to determine the upper bound of $ \hdim\bk $. The following four lemmas were initially proved by Fan, Schmeling and Troubetzkoy~\cite{FaScTr13} for the doubling map, and later by Liao and Sereut~\cite{LiaoSe13} in the context of Markov maps. 
We follow their ideas and demonstrate  more general results. In the Lemmas~\ref{l:multi-relation}--\ref{l:hitting time subset local dimension}, we will not assume that $ T $ is a Markov map.




\begin{lem}\label{l:multi-relation}
	Let $ T $ be a map on $ [0,1] $ and $ \nu $ be a $ T $-invariant exponential mixing measure. Let $ A_1,A_2,\dots,A_k $ be $ k $ subsets of $ [0,1] $ such that each $ A_i $ is a union of at most $ m $ disjoint balls. Then
	\[\prod_{i=1}^{k}\bigg(1-\frac{mC\beta^d}{\nu(A_i)}\bigg)\le \frac{\nu(A_1\cap T^{-d}A_2\cap\cdots\cap T^{-d(k-1)}A_k)}{\nu(A_1)\nu(A_2)\cdots\nu(A_k)}\le \prod_{i=1}^{k}\bigg(1+\frac{mC\beta^d}{\nu(A_i)}\bigg),\]
	where $ \beta $ is the constant appearing in~\eqref{eq:exponential mixing}.
\end{lem}
\begin{proof}
	Since each $ A_i $ is a union of at most $ m $ disjoint balls, the exponential mixing property of $ \nu $ gives that, for every $ d\ge 1 $,
	\begin{equation}\label{eq:exponential mixing multi}
		|\nu(A_i\cap T^{-d}B)-\nu(A_i)\nu(B)|\le mC\beta^d\nu(B),
	\end{equation}
	where $ B $ is a Borel measurable set.
	In particular, defining
	\[B_i=A_i\cap T^{-d}A_{i+1}\cap\cdots\cap T^{-d(k-i)}A_k,\]
	we get, for any $ i<k $,
	\[|\nu(A_i\cap T^{-d}(B_{i+1}))-\nu(A_i)\nu(B_{i+1})|\le mC\beta^d\nu(B_{i+1}).\]
	The above inequality can be written as
	\begin{equation*}
		1-\frac{mC\beta^d}{\nu(A_i)}\le \frac{\nu(A_i\cap T^{-d}B_{i+1})}{\nu(A_i)\nu(B_{i+1})}\le 1+\frac{mC\beta^d}{\nu(A_i)}.
	\end{equation*}
	Multiplying over all $ i\le k $ and using the identity
	\[B_{i+1}=A_{i+1}\cap T^{-d}B_{i+2},\]
	we have
	\[\prod_{i=1}^{k}\bigg(1-\frac{mC\beta^d}{\nu(A_i)}\bigg)\le \frac{\nu(A_1\cap T^{-d}A_2\cap\cdots\cap T^{-d(k-1)}A_k)}{\nu(A_1)\nu(A_2)\cdots\nu(A_k)}\le \prod_{i=1}^{k}\bigg(1+\frac{mC\beta^d}{\nu(A_i)}\bigg).\qedhere\]
\end{proof}
The following lemma illustrates that balls with small local dimension for exponential mixing measure are hit with big probability. 
\begin{lem}\label{l:big hitting probability}
	Let $ T $ be a map on $ [0,1] $ and $ \nu $ be a $ T $-invariant exponential mixing measure. Let $ h $ and $ \epsilon $ be two positive real numbers. For each $ n\in\N $, consider  $ N\le 2^n $ distinct balls $ B_1,\dots,B_N $ satisfying $ |B_i|=2^{-n} $ and $ \nu(B_i)\ge 2^{-n(h-\epsilon)} $ for all $ 1\le i\le N $. Set
	\[\mathcal C_{n,N,h}=\{x\in[0,1]:\exists 1\le i\le N\text{ such that }\tau(x,B_i)\ge 2^{nh}\}.\]
	Then there exists an integer $ n_h\in\N $ independent of $ N $ such that
	\[\text{for every }n\ge n_h,\quad \nu(\mathcal C_{n,N,h})\le 2^{-n}.\]
\end{lem}
\begin{proof}
	For each $ i\le N $, let
	\[\Delta_{i}:=\{x\in[0,1]:\forall k\le 2^{nh}, T^kx\notin B_i \}.\]
	Obviously we have $ \mathcal C_{n,N,h}=\bigcup_{i=1}^N\Delta_{i} $, so it suffices to bound from above each $ \nu(\Delta_{i}) $. Pick an integer $ \omega $ such that $ \omega>\log_{\beta^{-1}}2^{h} $. Let $ k=[2^{nh}/(\omega n)] $ be the integer part of $ 2^{nh}/(\omega n) $. Then
	\[\Delta_{i}\subset\bigcap_{j=1}^{k}\{x\in[0,1]: T^{j\omega n}x\notin B_i\}=\bigcap_{j=1}^{k}T^{-j\omega n}B_i^c.\]
	Since $ \omega>\log_{\beta^{-1}}2^{h}\ $, there is an $ n_h $ large enough such that for any $ n\ge n_h $,
\begin{equation}\label{eq:condition on nh}
	2C\beta^{\omega n}<2^{-nh-1}\le \nu(B_i)/2
\end{equation}
and
\begin{equation}\label{eq:condition 2 on nh}
	2^{n+1} \exp\bigg(\frac{-2^{n\epsilon}}{2\omega n}\bigg)\le 2^{-n}.
\end{equation}
	Now applying Lemma~\ref{l:multi-relation} to $ A_l= B_i $ for all $ l\le N $ and to $ m=2 $, we conclude from~\eqref{eq:condition on nh} that
	\begin{align*}
		\nu(\Delta_{i})\le\nu(\cap_{j=1}^{k}T^{-j\omega n}B_i^c)
		&\le (\nu(B_i^c)+2C\beta^{\omega n})^{k}\\
		&\le (1-\nu(B_i)/2)^{k}\\
		&\le (1-2^{-n(h-\epsilon)-1})^{ 2^{nh}/(\omega n)-1}\\
		&=(1-2^{-n(h-\epsilon)-1})^{-1} \exp\bigg(\frac{2^{nh}\log(1-2^{-n(h-\epsilon)-1})}{\omega n}\bigg)\\
		&\le 2\exp\bigg(\frac{-2^{n\epsilon}}{2\omega n}\bigg).
	\end{align*}
By~\eqref{eq:condition 2 on nh}, 
	\[\nu(C_{n,N,h})\le \sum_{i=1}^{N}\nu(B_i)\le 2^{n+1} \exp\bigg(\frac{-2^{n\epsilon}}{2\omega n}\bigg)\le 2^{-n}.\qedhere\]
\end{proof}
Let us recall that $ \{y:\ovr(x,y)\ge s\} $ is random depending the random element $ x $, but $ \{y:\od_{\uph}(y)\ge s\} $ is independent of $ x $. The following lemma reveals a connection between the deterministic set $ \{y:\ovr(x,y)\ge s\} $ and the random set $ \{y:\od_{\uph}(y)\ge s\} $.

\begin{lem}\label{l:hitting time subset local dimension}
	Let $ T $ be a map on $ [0,1] $ and $ \nu $ be a $ T $-invariant exponential mixing measure. Let $ s\ge 0 $. Then for $ \nu $-a.e.\,$ x $,
	\[\{y:\ovr(x,y)\ge s\}\subset\{y:\od_{\nu}(y)\ge s\}.\]
\end{lem}
\begin{proof}
	
	The case $ s=0 $ is obvious. We therefore assume $ s>0 $. For any integer $ n\ge 1 $, let
	\[\mathcal R_{n,s,\epsilon}(x)=\big\{y:\tau\big(x,B(y,2^{-n+1})\big)\ge 2^{n(s-\epsilon)}\big\},\quad\mathcal E_{n,s,\epsilon}=\big\{y:\nu\big(B(y,2^{-n})\big)\le 2^{-n(s-2\epsilon)}\big\}.\]
	
	By definition, $ y\in \{y:\ovr(x,y)\ge s\} $ if and only if for any $ \epsilon>0 $, there exist infinitely many integers $ n $ such that
	\[\frac{\log \tau(x,B(y,2^{-n+1}))}{\log 2^n}\ge s-\epsilon.\]
	Hence, we have
	\begin{equation}\label{eq:ovr>s=}
		\{y:\ovr(x,y)\ge s\}=\bigcap_{\epsilon>0}\limsup_{n\to\infty}\mathcal R_{n,s,\epsilon}(x).
	\end{equation}
	Similarly,
	\[\{y:\od_{\uph}(y)\ge s\}=\bigcap_{\epsilon>0}\limsup_{n\to\infty}\mathcal E_{n,s,\epsilon}.\]
	Thus, it is sufficient to prove that, for $ \nu $-a.e.\,$ x $, there exists some integer $ n(x) $ such that
	\begin{equation}\label{eq:Rnse subset Ense}
		\text{for all }n\ge n(x),\quad\mathcal R_{n,s,\epsilon}(x)\subset \mathcal E_{n,s,\epsilon},
	\end{equation}
	or equivalently,
	\begin{equation}\label{eq:Rnsec subset Ensec}
		\text{for all }n\ge n(x),\quad\mathcal E^c_{n,s,\epsilon}(x)\subset \mathcal R^c_{n,s,\epsilon}.
	\end{equation}

	Notice that $ \mathcal E_{n,s,\epsilon}^c $ can be covered by $ N\le 2^n $ balls with center in $ \mathcal E_{n,s,\epsilon}^c $ and radius $ 2^{-n} $. Let $ \mathcal F_{n,s,\epsilon}:=\{B_1,B_2,\dots,B_N\} $ be the collection of these balls. By definition, we have $ \nu(B_i)\le 2^{-n(s-2\epsilon)} $. Applying Lemma~\ref{l:big hitting probability} to the collection $ \mathcal F_{n,s,\epsilon} $ of balls and to $ h=s-\epsilon $, we see that
	\[\sum_{n\ge n_h}\nu(\{x:\exists B\in\mathcal F_{n,s,\epsilon} \text{ such that }\tau(x,B)\ge 2^{n(s-\epsilon)}\})\le \sum_{n\ge n_h}2^{-n}<\infty.\]
	By Borel-Cantelli Lemma, for $ \nu $-a.e.\,$ x $, there exists an integer $ n(x) $ such that
	\[\text{for all }n\ge n(x),\text{ for all }B\in \mathcal F_{n,s,\epsilon},\quad \tau(x, B)< 2^{n(s-\epsilon)}.\]
	If $ y\in B $ for some $ B\in \mathcal F_{n,s,\epsilon} $ and $ n\ge n(x) $, then $ B\subset B(y,2^{-n+1}) $, which implies that $ \tau\big(x, B(y,2^{-n+1})\big)<\tau(x,B)  $. We then deduce that $ B $ is included in $ \mathcal R_{n,s,\epsilon}^c $. This yields $ \mathcal E_{n,s,\epsilon}^c\subset \mathcal R_{n,s,\epsilon}^c $, which is what we want.
\end{proof}
\begin{rem}
	With the notation in Lemma~\ref{l:hitting time subset local dimension}, proceeding with the same argument as~\eqref{eq:ovr>s=}, we have
	\[	\{y:\ur(x,y)\ge s\}=\bigcap_{\epsilon>0}\liminf_{n\to\infty}\mathcal R_{n,s,\epsilon}(x)\quad \text{and}\quad\{y:\ud_{\uph}(y)\ge s\}=\bigcap_{\epsilon>0}\liminf_{n\to\infty}\mathcal E_{n,s,\epsilon}.\]
	It then follows from~\eqref{eq:Rnse subset Ense} that for $ \nu $-a.e. $ x $,
	\[\{y:\ur(x,y)\ge s\}\subset\{y:\ud_{\nu}(y)\ge s\}.\]
\end{rem}
Applying Lemma~\ref{l:hitting time subset local dimension} to the Gibbs measure $ \uph $, we get the following upper bound.
\begin{lem}\label{l:upper bound for hdimfk}
	Let $ T $ be a Markov map. Let $ \phi $ be a H\"older continuous potential and $ \uph $ be the associated Gibbs measure. Let $ 1/\kappa\ge\alpha_{\max} $, then for $ \uph $-a.e.\,$ x $, 
	\[\hdim\bk\le D_{\uph}(1/\kappa).\]
	Moreover, if $ 1/\kappa>\alpha_+ $, then for $ \uph $-a.e.\,$ x $,
	\[\bk=\emptyset.\]
\end{lem}
\begin{proof}
	Recall that Lemma~\ref{l:described by hitting time} asserts that
	\[\bk\subset\{y:\ovr(x,y)\ge 1/\kappa\}\cup\mathcal O^+(x).\]
	A direct application of Proposition~\ref{l:dimension spectrum} and Lemma~\ref{l:hitting time subset local dimension} yields the first conclusion.
	
	The second conclusion follows from Lemmas~\ref{l:>alpha+ empty} and~\ref{l:described by hitting time}.
\end{proof}

Collecting the results obtained in this subsection, we can prove Theorem~\ref{t:main} except the item (3).
\begin{proof}[Proof of the items (1), (2) and (4) of Theorem~\ref{t:main}]
	Combining with Lemmas~\ref{l:lower bound for hdimfk} and~\ref{l:upper bound for hdimfk}, we get the desired result.
\end{proof}
\subsection{The study of $ \uk $}
In this subsection, we prove the remaining part of Theorem~\ref{t:main}, that is, item (3). We begin by proving the lower bound of $ \hdim\uk $, which may be proved in much the same way as Lemma~\ref{l:lower bound for hdimfk}.

\begin{lem}
	Let $ T $ be a Markov map. Let $ \phi $ be a H\"older continuous potential and let $ \uph $ be the corresponding Gibbs measure.
	\begin{enumerate}[\upshape(a)]
		\item If $ 1/\kappa\in(0,\alpha_{\max}]\setminus\{\alpha_-\} $, then $ \hdim\uk\ge D_{\uph}(1/\kappa) $ for $ \uph $-a.e.\,$ x $. 
		\item If $ 1/\kappa\in(\alpha_{\max},+\infty) $, then $ \hdim\uk=1 $ for $ \uph $-a.e.\,$ x $.
	\end{enumerate}
\end{lem}

\begin{proof}
	(a) By Lemma~\ref{l:dimension spectrum}, the Hausdorff dimension of the level set $ \{y:d_{\uph}(y)=1/\kappa\} $ is zero if $ 1/\kappa<\alpha_- $. Therefore $ \hdim\uk\ge 0=D_{\uph}(1/\kappa) $. 
	
	The remaining case $ 1/\kappa\in(\alpha_-,\alpha_{\max}] $ holds by the same reason as Lemma~\ref{l:lower bound for hdimfk} ($ b $).
	
	(b) Observe that the full Lebesgue measure statement implies the full Hausdorff dimension statement, it follows from item (2) of Theorem~\ref{t:main} that $ \hdim\uk= 1 $ when $ 1/\kappa\in(\alpha_{\max},+\infty) $.
\end{proof}

It is left to show the upper bound of $ \hdim\uk $ when $ 1/\kappa\le\alpha_{\max} $. The proof combines the methods developed in~\cite[\S 7]{FaScTr13} and~\cite[Theorem 8]{KoLiPe21}. 
Heuristically, the larger the local dimension of a point is, the less likely it is to be hit.
\begin{lem}\label{l:upper bound of hdimuk}
	Let $ T $ be a Markov map. Let $ \phi $ be a H\"older continuous potential and $ \uph $ be the associated Gibbs measure. Let $ 1/\kappa\le\alpha_{\max} $, then for $ \uph $-a.e.\,$ x $,
	\[\hdim\uk\le D_{\uph}(1/\kappa).\]
\end{lem}

\begin{proof}
	The proof will be divided into two steps.
	
	Step 1. Given any $ a>1/\kappa $, we are going to prove that
	\begin{equation}\label{eq:ukcap}
		\hdim \big(\uk\cap\{y:\ud_{\uph}(y)>a\}\big)=0\quad\text{for }\mu_\phi\text{-}a.e.\,x.
	\end{equation}
	Suppose now that~\eqref{eq:ukcap} is established. Let $ (a_m)_{m\ge 1} $ be a monotonically decreasing sequence of real numbers converging to $ 1/\kappa $. Applying ~\eqref{eq:ukcap} to each $ a_m $ yields a full $ \uph $-measure set corresponding to $ a_m $. Then by taking the intersection of these countable full $ \uph $-measure sets, we conclude from the countable stability of Hausdorff dimension that
	\[\hdim \big(\uk\cap\{y:\ud_{\uph}(y)>1/\kappa\}\big)=0\quad\text{for }\mu_\phi\text{-}a.e.\,x.\]
	As a result, by Lemma~\ref{l:dimension spectrum}, for $ \uph $-a.e.\,$ x $,
	\begin{align*}
		\hdim \uk&=\hdim \big(\uk\cap\big(\{y:\underline{d}_{\mu_\phi}(y)\le 1/\kappa\}\cup \{y:\ud_{\uph}(y)>1/\kappa\}\big)\big)\notag\\
		&=\hdim \big(\uk\cap\{y:\underline{d}_{\mu_\phi}(y)\le 1/\kappa\}\big)\notag\\
		&\le\hdim\{y:\underline{d}_{\mu_\phi}(y)\le 1/\kappa\}\label{eq:upper bound of hdimuk}=D_{\uph}(1/\kappa).
	\end{align*}
	This clearly yields the lemma.
	
	Choose $ b\in(1/\kappa, a) $. Put $ A_n:=\{y:\mu_\phi(B(y, r))<r^b\text{ for all }r<2^{-n}\} $. By the definition of $ \ud_{\uph}(y) $, we have
	\[\{y:\ud_{\uph}(y)>a\}\subset\bigcup_{n=1}^\infty A_n.\]
	Thus,~\eqref{eq:ukcap} is reduced to show that for any $ n\ge 1 $,
	\begin{equation}\label{eq:ukcapan}
		\hdim (\uk\cap A_n)=0\quad\text{for }\mu_\phi\text{-}a.e.\,x.
	\end{equation}

	Step 2. The next objective is to prove~\eqref{eq:ukcapan}.
	
	Fix $ n\ge 1 $. Let $ \epsilon>0 $ be arbitrary. Choose a large integer $ l\ge n $ with
	\begin{equation}\label{eq:condition l}
	12\times 2^{-\kappa l}<2^{-n}\quad\text{and}\quad	\gamma^312^b2^{(1-b\kappa)l}< \epsilon,
	\end{equation}
where the constant $ \gamma $ is defined in~\eqref{eq:quasi-Bernoulli}.
	Let $ \theta_j=[\kappa j\log_{L_1}2]+1 $, where $ L_1 $ is given in~\eqref{eq:length of basic interval}. Then, by~\eqref{eq:length of basic interval} the length of each basic interval of generation $ \theta_j $ is smaller than $ 2^{-\kappa j} $. Define
	\[\mathcal I_{j}(x):=\bigcup_{J\in\Sigma_{\theta_j}\colon d(I_{\theta_j}(x),J)<2^{-\kappa j}}J,\]
	where $ d(\cdot,\cdot) $ is the Euclidean metric. Clearly $ \mathcal I_j(x) $ covers the ball $ B(x, 2^{-\kappa j}) $ and is contained in $ B(x, 3\times2^{-\kappa j}) $. 
	Moreover, if $ I_{\theta_j}(x)=I_{\theta_j}(y) $, then $ \mathcal I_j(x)=\mathcal I_j(y) $. With the notation $ \mathcal I_j(x) $, we consider the set
	\[G_{l,i}(x)= A_n\cap\bigg(\bigcap_{j=l}^i\bigcup_{k=1}^{2^j}\mathcal{I}_{j}(T^kx)\bigg).\]
	The advantage of using $ \mathcal I_j(x) $ rather than $ B(x, 2^{-\kappa j}) $ is that the map $ x\mapsto G_{l,i}(x) $ is constant on each basic interval of generation $ 2^i+\theta_i $. We are going to construct inductively a cover of $ G_{l,i}(x) $ by the family $ \big\{B(T^kx, 3\times2^{-\kappa i}):k\in S_i(x)\big\} $ of balls, where $ S_i(x)\subset \{1,2,\dots,2^i\} $. 

	For $ i=l $, we let $ S_l(x)\subset\{1,2,\dots, 2^l\} $ consist of those $ k\le 2^l $ such that $ \mathcal I_l(T^kx) $ intersects $ A_n $. Suppose now that $ S_i(x) $ has been defined. We define $ S_{i+1}(x) $ to consist of those $ k\le 2^{i+1} $ such that $ \mathcal I_{i+1}(T^kx) $ intersects $ G_{l,i}(x) $. Then the family $ \big\{B(T^kx, 3\times2^{-\kappa (i+1)}):k\in S_{i+1}(x)\big\} $ of balls forms a cover of $ G_{l,i+1}(x) $, and the construction is completed. 
	
	With the aid of the notation $ \mathcal I_{i+1}(T^kx) $, one can verify that $ x\mapsto S_{i+1}(x) $ is constant on each basic interval of generation $ 2^{i+1}+\theta_{i+1} $. Let $ N_{i+1}(x):=\sharp S_{i+1}(x) $, then $ x\mapsto N_{i+1}(x) $ is also constant on each basic interval of generation $ 2^{i+1}+\theta_{i+1} $. 
	
	In order to establish~\eqref{eq:ukcapan}, we need to estimate $ N_{i+1}(x) $. For those $ k\in S_{i+1}(x)\cap\{1,2,\dots,2^i\} $, since $ \mathcal I_{i+1}(T^kx)\subset\mathcal I_i(T^kx) $ and $ G_{l,i}(x)\subset G_{l,i-1}(x) $,  we must have that $ \mathcal I_i(T^kx) $ intersects $ G_{l,i-1}(x) $, hence $ k\in S_i(x) $. On the other hand, since $ \mathcal I_{i+1}(T^kx) $ is contained in $ B(T^kx,3\times2^{-\kappa (i+1)}) $, if $ \mathcal I_{i+1}(T^kx) $ has non-empty intersection with $ G_{l,i}(x) $, then the distance between $ T^kx $ and $ G_{l,i}(x) $ is less than $ 3\times2^{-\kappa (i+1)} $. In particular, 
	\begin{equation}\label{eq:Gi(x)}
		T^kx \in\big\{y: d\big(y, G_{l,i}(x)\big)<3\times2^{-\kappa(i+1)}\big\}\subset \bigcup_{J\in\Sigma_{\theta_i}\colon d(J,G_{l,i}(x))<3\times2^{-\kappa (i+1)}}J.
	\end{equation}
	Denote the right hand side union as $ \hat G_{l,i}(x) $. The set $ \hat G_{l,i}(x) $ is nothing but the union of cylinders of level $ \theta_i $ whose distance from $ G_{l,i}(x) $ is less than $ 3\times 2^{-\kappa(i+1)} $. Thus, by the fact that $ x\mapsto G_{l,i}(x) $ is constant on each basic interval  of generation $ 2^i+\theta_i $, we have
	\begin{equation}\label{eq:Gi(x)=Gi(y)}
		\hat G_{l,i}(x)=\hat G_{l,i}(y),\quad \text{whenever}\quad I_{2^i+\theta_i}(x)=I_{2^i+\theta_i}(y) .
	\end{equation}
	According to the above discussion, it holds that
	\[N_{i+1}(x)\le N_i(x)+M_{i+1}(x),\]
	where $ M_{i+1}(x) $ is the number of $ 2^i<k\le 2^{i+1} $ for which $ T^k x $ intersects $ \hat G_{l,i}(x) $.
	
	 The function $ M_{i+1}(x) $ can further be written as:
	\begin{equation*}
		M_{i+1}(x)=\sum_{k=2^i+1}^{2^{i+1}}\chi_{\hat G_{l,i}(x)}(T^kx)=\sum_{k=2^i+1}^{2^i+\theta_i}\chi_{\hat G_{l,i}(x)}(T^kx)+\sum_{k=2^i+\theta_i+1}^{2^{i+1}}\chi_{\hat G_{l,i}(x)}(T^kx).
	\end{equation*}	
	It then follows from the locally constant property of $ \hat G_{l,i(x)} $~\eqref{eq:Gi(x)=Gi(y)} that
	\begin{align}
		\notag\int M_{i+1}(x) d\uph(x)&=\sum_{k=2^i+1}^{2^i+\theta_i}\int \chi_{\hat G_{l,i}(x)}(T^kx)d\uph(x)+\sum_{k=2^i+\theta_i+1}^{2^{i+1}}\int \chi_{\hat G_{l,i}(x)}(T^kx)d\uph(x)\\
		&\le \theta_i+\sum_{J\in\Sigma_{2^i+\theta_i}}\sum_{k=2^i+\theta_i+1}^{2^{i+1}}\int \chi_J(x) \chi_{\hat G_{l,i}(x)}(T^kx)d\uph(x)\notag\\
		&=\theta_i+\sum_{J\in\Sigma_{2^i+\theta_i}}\sum_{k=2^i+\theta_i+1}^{2^{i+1}}\int \chi_J(x) \chi_{\hat G_{l,i}(x_J)}(T^kx)d\uph(x),\label{eq:sumsum}
	\end{align}
	where $ x_J $ is any fixed point of $ J $.
Now the task is to deal with the right hand side summation. 
We deduce from the quasi-Bernoulli property of $ \uph $~\eqref{eq:quasi-Bernoulli consequence} that for each $ k>2^i+\theta_i $,
\begin{equation}\label{eq:int<CJG}
	\int \chi_J(x) \chi_{\hat G_{l,i}(x_J)}(T^kx)d\uph(x)=\uph\big(J\cap T^{-k}\big(\hat G_{l,i}(x_J)\big)\big)\le \gamma^3\uph(J)\uph\big(\hat{G}_{l,i}(x_J)\big).
\end{equation}
Since $ G_{l,i}(x_J) $ can be covered by the family $ \big\{B(T^kx_J, 3\times2^{-\kappa i}):k\in S_i(x_J)\big\} $ of balls, then by~\eqref{eq:Gi(x)} the family $ \mathcal F_i(x_J):=\big\{B(T^kx_J, 6\times2^{-\kappa i}):k\in S_i(x_J)\big\} $ of enlarged balls forms a cover of $ \hat{G}_{l,i}(x_J) $. Observe that each enlarged ball $ B\in \mathcal F_i(x_J) $ intersects $ A_n $, thus $ B\subset B(y,12\times 2^{-\kappa i}) $ for some $ y\in A_n $. Then by the definition of $ A_n $ and~\eqref{eq:condition l}, 
\[\uph(B)\le \uph(B(y,12\times 2^{-\kappa i}))\le12^b2^{-b\kappa i}.\] Accordingly,
\begin{equation}\label{eq:uphhatG}
	\uph\big(\hat{G}_{l,i}(x_J)\big)\le 12^b2^{-b\kappa i}N_i(x_J).
\end{equation}

Recall that $ x\mapsto N_i(x) $ is constant on each basic interval of generation $ 2^i+\theta_i $. Applying the upper bound~\eqref{eq:uphhatG} on $ \uph\big(\hat{G}_{l,i}(x_J)\big) $ to~\eqref{eq:int<CJG}, and then substituting~\eqref{eq:int<CJG} into~\eqref{eq:sumsum}, we have
	\begin{align*}
	\notag\int M_{i+1}(x) d\uph(x)
	&\le \theta_i+\sum_{J\in\Sigma_{2^i+\theta_i}}\sum_{k=2^i+\theta_i+1}^{2^{i+1}}\gamma^3\uph\big(\hat{G}_{l,i}(x_J)\big)\uph(J)\\
	&\le \theta_i+\sum_{J\in\Sigma_{2^i+\theta_i}}\sum_{k=2^i+\theta_i+1}^{2^{i+1}}\gamma^312^b2^{-b\kappa i}N_i(x_J)\uph(J)\\
	&= \theta_i+\gamma^312^b2^{-b\kappa i}(2^i-\theta_i)\int N_i(x)d\uph(x)\\
	&\le \theta_i+\epsilon\int N_i(x)d\uph(x),
	\end{align*}
	where the last inequality follows from~\eqref{eq:condition l}.
	Since $ N_{i+1}(x)\le N_i(x)+M_{i+1}(x) $, we have
	\begin{equation}\label{eq:intNi+1x}
		\int N_{i+1}(x)d\uph(x)\le \theta_i+(1+\epsilon)\int N_i(x)d\uph(x).
	\end{equation}
  Note that~\eqref{eq:intNi+1x} holds for all $ i\ge l $ and $ N_l(x)\le 2^l $. Then
	\[\begin{split}
		\int N_{i+1}(x)d\uph(x)&\le \sum_{k=l}^{i}(1+\epsilon)^{i-k}\theta_k+(1+\epsilon)^{i-l+1}\int N_l(x)d\uph(x)\\
		&<(1+\epsilon)^i(i\theta_i+2^l).
	\end{split}\]	
	By Markov's inequality,
	\[\begin{split}
		\uph\big(\big\{x:N_{i+1}(x)\ge (1+\epsilon)^{2i}(i\theta_i+2^l)\big\}\big)&\le \uph\bigg(\bigg\{x:N_{i+1}(x)\ge (1+\epsilon)^i\int N_{i+1}(x)d\uph(x)\bigg\}\bigg)\\
		&\le (1+\epsilon)^{-i},
	\end{split}\]
	which is summable over $ i $. Hence, for $ \uph $-a.e.\,$ x $, there is an $ i_0(x) $ such that
	\begin{equation}\label{eq:Ni+1(x)le (1+epsilon)}
		N_{i+1}(x)\le (1+\epsilon)^{2i}(i\theta_i+2^l)
	\end{equation}
	holds for all $ i\ge i_0(x) $. 
	
	Denote by $ F_{l,\epsilon} $ the full measure set on which the formula~\eqref{eq:Ni+1(x)le (1+epsilon)} holds. Let $ x\in F_{l,\epsilon} $. Then such an $ i_0(x) $ exists. With $ i\ge i_0(x) $, we may cover the set
	\[G_l(x)=A_n\cap\bigg(\bigcap_{j=l}^\infty\bigcup_{k=1}^{2^j}\mathcal I_j(T^kx)\bigg)\]
	by $ N_i(x) $ balls of radius $ 3\times2^{-\kappa i} $. Since $ \theta_i=[\kappa i\log_{L_1}2]+1\le ci $ for some $ c>0 $, we have	\begin{equation}\label{eq:upper bound of Gl(x)}
		\hdim G_l(x)\le\limsup_{i\to\infty}\frac{\log N_i(x)}{\log 2^{\kappa i}}\le\limsup_{i\to\infty}\frac{\log \big((1+\epsilon)^{2i}(i\theta_i+2^l)\big)}{\log 2^{\kappa i}}=\frac{2\log (1+\epsilon)}{\kappa\log2}.
	\end{equation}
	Let $ (\epsilon_m)_{m\ge 1} $ be a monotonically decreasing sequence of real numbers converging to $ 0 $. For each $ \epsilon_m $, choose an integer $ l_m $ satisfying~\eqref{eq:condition l}, but with $ \epsilon $ replaced by $ \epsilon_m $. For every $ m\ge 1 $, by the same reason as~\eqref{eq:upper bound of Gl(x)}, there exists a full $ \uph $-measure set $ F_{l_m,\epsilon_m} $ such that
	\[\text{for all }x\in F_{l_m,\epsilon_m},\quad\hdim G_{l_m}(x)\le\frac{2\log (1+\epsilon_m)}{\kappa\log2}.\]
	By taking the intersection of the countable full $ \uph $-measure sets $ (F_{l_m,\epsilon_m})_{m\ge 1} $, and using the fact that $ G_l(x) $ is increasing in $ l $, we obtain that for $ \uph $-a.e.\,$ x $,
	\[\text{for any }l\ge 1,\quad\hdim G_l(x)\le\lim_{m\to \infty}\hdim G_{l_m}(x)=0.\]
	
	We conclude~\eqref{eq:ukcapan} by noting that
	\[\uk\cap A_n\subset\bigcup_{l\ge 1}G_l(x).\qedhere\]
\end{proof}
\begin{rem}
	Recall that Lemma~\ref{l:hitting time subset local dimension} exhibits a relation between $ \uk $ and hitting time:
	\[	\{ y:\ovr(x,y)<1/\kappa\}\setminus\mathcal O^+(x)\subset\uk\subset\{ y:\ovr(x,y)\le1/\kappa\}.\]
	With this relation in mind, it is natural to investigate the size of the level sets
	\[\{y:R(x,y)=1/\kappa\},\quad \kappa\in(0,\infty).\]
	
	Lemmas~\ref{l:hitting time subset local dimension} and~\ref{l:upper bound of hdimuk} together with the inclusions
	\[\{y:R(x,y)=1/\kappa\}\subset\{y:\ovr(x,y)\le 1/\kappa\}\quad\text{and}\quad \{y:R(x,y)=1/\kappa\}\subset\{y:\ovr(x,y)\ge 1/\kappa\}\]
	 give the upper bound for Hausdorff dimension:
	\[\hdim\{y:R(x,y)=1/\kappa\}\le D_{\uph}(1/\kappa),\quad\text{for }\uph\text{-}a.e.\, x. \]
	The lower bound, coinciding with the upper bound, can be proved by the same argument as Lemma~\ref{l:lower bound for hdimfk}. Thus for $ \uph $-a.e.\,$ x $,
	\[\hdim\{y:R(x,y)=1/\kappa\}=D_{\uph}(1/\kappa),\quad\text{if }1/\kappa\notin \{\alpha_-,\alpha_+\}. \]
\end{rem}
{\bf Acknowledgements} The authors would like to thank Tomas Persson for pointing out that the Lebesgue measure statement in Example 1 can be concluded from the reference~\cite{HKKP21} by a Fubini based argument.


\begin{thebibliography}{10}
	
	\bibitem{AlBe98}
	P. Alessandri and V. Berth\'e. Three distance theorems and combinatorics on words. {\em Enseignement Math.} 44 (1998), 103--132.
	
	\bibitem{Bal00}
	V. Baladi. 
	\newblock {\em Positive Transfer Operators and Decay of Correlations}. 
	\newblock (Advanced Series in Nonlinear
	Dynamics, 16). 
	\newblock World Scientific, River Edge, NJ, 2000.
	
	\bibitem{BaFa05}
	J. Barral and A. H. Fan. Covering numbers of different points in Dvoretzky covering. {\em Bull. Sci.
	Math. Fr.} 129 (2005), 275--317.
	
	\bibitem{BaPeS97}
	L. Barreira, Y. Pesin and J. Schmeling. On a general concept of multifractality: multifractal spectra for
	dimensions, entropies, and Lyapunov exponents. Multifractal rigidity. {\em Chaos} 7 (1997), 27--38.
	
	
	\bibitem{BiPe17}
	C. J. Bishop and Y. Peres.
	{\em Fractals in probability and analysis.}
	Cambridge Studies in Advanced Mathematics, 162. Cambridge University Press, Cambridge, 2017.
	
	\bibitem{Bow75}
	R. Bowen. {\em Equilibrium States and the Ergodic Theory of Anosov Diffeomorphisms.} Springer, Berlin,
	1975.
	
	\bibitem{BrMiP92}
	G. Brown, G. Michon and J. Peyri\`ere. On the multifractal analysis of measures. {\em J. Stat. Phys.} 66 (1992), 775--790.
	
	\bibitem{Bug03}
	Y. Bugeaud.
	\newblock A note on inhomogeneous Diophantine approximation.
	\newblock {\em Glasgow Math. J.} 45 (2003), 105--110.
	
	\bibitem{BuLi16}
	Y.	Bugeaud and L. Liao. Uniform Diophantine approximation related to $ b $-ary and $ \beta $-expansions. {\em Ergod. Th. $ \& $ Dynam. Sys.} 36 (2016), 1--22.
	
	\bibitem{CoLeP87}
	P. Collet, J. Lebowitz and A. Porzio. The dimension spectrum of some dynamical systems. {\em J. Stat. Phys.}
	47 (1987), 609--644.
	
	\bibitem{Dvo56}
	A. Dvoretzky. On covering the circle by randomly placed arcs. {\em Pro. Nat. Acad. Sci. USA} 42 (1956),
	199--203.
	
	
	\bibitem{Fan02}
	A. H. Fan. How many intervals cover a point in Dvoretzky covering? {\em Israel J. Math.} 131 (2002), 157--184.
	
	
	
	
	
	\bibitem{FaScTr13}
	A. H. Fan, J. Schmeling and S. Troubetzkoy.
	\newblock A multifractal mass transference principle for Gibbs measures with applications to dynamical Diophantine approximation.
	\newblock {\em Proc. London Math. Soc.} 107 (2013), 1173--1219.
	
	
	\bibitem{Gal07}
	S. Galatolo. 
	Dimension and hitting time in rapidly mixing systems. {\em Math. Res. Lett.} 14 (2007), 797--805.
	
	\bibitem{HKKP21}
	M. Holland, M. Kirsebom, P. Kunde and T. Persson. Dichotomy results for eventually always hitting time statistics and almost sure growth of extremes. arXiv:2109.06314, Sep 2021.
	
	\bibitem{Jen06}
	O. Jenkinson. Ergodic optimization. {\em Discrete Contin. Dyn. Syst.} 15 (2006), 197--224.
	
	\bibitem{JoSt08}
	J. Jonasson and J. Steif. Dynamical models for circle covering: Brownian motion and Poisson updating.
	{\em Ann. Probab.} 36 (2008), 739--764.
	
	\bibitem{KimLi19}
	D. H. Kim and L. Liao.
	\newblock Dirichlet uniformly well-approximated numbers.
	\newblock {\em Int. Math. Res.
		Not.} 24 (2019), 7691--7732.
	
	\bibitem{KoLiPe21}
	H. Koivusalo, L. Liao and T. Persson.
	\newblock Uniform random covering problems.
	\newblock {\em Int. Math. Res.
	Not.} pages 1--27, October 2021.
	\newblock Published online. DOI: 10.1093/imrn/rnab272.
	
	\bibitem{LiaoSe13}
	L. Liao and S. Seuret.
	\newblock Diophantine approximation by orbits of expanding Markov maps.
	\newblock {\em Ergod. Th. $ \& $ Dynam. Sys.} 33 (2013), 585--608.
	
	\bibitem{LiSaV98}
	C. Liverani, B. Saussol and S. Vaienti. \newblock Conformal measure and decay of correlation for covering weighted systems.
	\newblock {\em Ergod. Th. $ \& $ Dynam. Sys.} 18 (1998), 1399--1420.
	
	\bibitem{PaPo90}
	\newblock W. Parry and M. Pollicott. Zeta functions and the periodic orbit structure of hyperbolic dynamics.
	{\em Ast\'erisque} 1990, 187--188.
	
	\bibitem{PerRams17}
	T. Persson and M. Rams.
	\newblock On shrinking targets for piecewise expanding interval maps.
	\newblock {\em Ergod. Th. $ \& $ Dynam. Sys.} 37 (2017), 646--663.
	
	\bibitem{PeWe97}
	Y. Pesin and H. Weiss. The multifractal analysis of Gibbs measures: motivation, mathematical foundation,
	and examples. {\em Chaos} 7 (1997), 89--106.
	
	\bibitem{Ran89}
	D. A. Rand. 
	The singularity spectrum $ f(\alpha) $ for cookie-cutters. {\em Ergod. Th. $ \& $ Dynam. Sys.} 9 (1989), 527--541.
	
	
	\bibitem{Rue04}
	D. Ruelle. {\em Thermodynamic formalism. The Mathematical Structures of Equilibrium Statistical
	Mechanics}. Cambridge University Press, Cambridge, 2nd edition, 2004.

	\bibitem{ScTr03}
	J. Schmeling and S. Troubetzkoy.
	\newblock Inhomogeneous Diophantine approximation and angular recurrence properties of the billiard flow in certain polygons.
	{\em Mat. Sb.} 194(2) (2003) 129--144; translation in 
	{\em Mat. Sb.} 194(2) (2003), 295--309.
	
	\bibitem{Seu18}
	S. Seuret. 
	Inhomogeneous random coverings 	of topological Markov shifts. {\em Math. Proc. Camb. Phil. Soc.} 165 (2018), 341--357.
	
	\bibitem{She72}
	L. Shepp. Covering the circle with random arcs. {\em Israel J. Math.} 11 (1972), 328--345.
	
	\bibitem{Sim94}
	D. Simpelaere. Dimension spectrum of Axiom A diffeomorphisms. II. Gibbs measures. {\em J. Stat. Phys.}
	76 (1994), 1359--1375.
	
	\bibitem{Tan15}
	J. M. Tang. 
	Random coverings of the circle with i.i.d. centers. {\em Sci. China Math.} 55 (2015), 1257--1268.
	
	\bibitem{Wal78}
	P. Walters. Invariant measures and equilibrium states for some mappings which expand distances. {\em Trans. Amer. Math. Soc.} 236 (1978), 121--153.
		
	\bibitem{ZhWu20}
L. Zheng and M. Wu. 
Uniform recurrence properties for beta-transformation. {\em Nonlinearity} 33 (2020), 4590--4612. 

\end{thebibliography}
\end{document}